\documentclass[table]{amsart} 
\usepackage{amsmath}
\usepackage[margin=1.0in]{geometry}
\usepackage{amscd,amsmath,amsxtra,amsthm,amssymb,stmaryrd,xr,mathrsfs,mathtools, enumerate, commath, comment, enumitem, graphicx}
\usepackage{amsfonts, amssymb, amsthm, amsmath, braket, hyperref, mathtools, comment, mathrsfs, enumitem, cleveref}
\usepackage{stmaryrd}
\usepackage{amsfonts}
\usepackage{orcidlink}
\usepackage{multirow}
\usepackage{xcolor}
\usepackage{commath}
\usepackage{float}
\usepackage{comment}
\usepackage{tikz-cd}
\usepackage{tkz-graph}
\usepackage{colonequals}
\usepackage{hyperref}
\usepackage{graphicx}
\usepackage{bigstrut}
\usepackage{tabularx}
\usepackage{longtable, multirow, array}
\usepackage{caption}
\usepackage{tikz}
\usepackage{blindtext}
\usepackage{adjustbox}
\allowdisplaybreaks

\usepackage{longtable} 
\usepackage{pdflscape} 
\usepackage{booktabs}
\usepackage{hyperref}
\definecolor{vegasgold}{rgb}{0.77, 0.7, 0.35}
\definecolor{darkgoldenrod}{rgb}{0.72, 0.53, 0.04}
\definecolor{gold(metallic)}{rgb}{0.83, 0.69, 0.22}
\hypersetup{
 colorlinks=true,
 linkcolor=darkgoldenrod,
 filecolor=brown,      
 urlcolor=gold(metallic),
 citecolor=darkgoldenrod,
 }

\usepackage[all,cmtip]{xy}
\DeclareFontFamily{U}{wncy}{}
\DeclareFontShape{U}{wncy}{m}{n}{<->wncyr10}{}
\DeclareSymbolFont{mcy}{U}{wncy}{m}{n}
\DeclareMathSymbol{\Sh}{\mathord}{mcy}{"58}
\usepackage[T2A,T1]{fontenc}
\usepackage[OT2,T1]{fontenc}

\usepackage{tikz}
\usetikzlibrary{shapes.geometric}
\usetikzlibrary{decorations.markings}
\tikzset{every loop/.style={min distance=10mm,looseness=10}}
\tikzstyle{vertex}=[auto=left,circle,minimum size=1pt,inner sep=0pt]

\newtheorem{thm}{Theorem}[section]

\newtheorem{prop}[thm]{Proposition}
\newtheorem{cor}[thm]{Corollary}
\theoremstyle{definition}

\newtheorem{rmk}[thm]{Remark}

\newcommand{\Z}{\mathbb{Z}}

\newcommand{\Q}{\mathbb{Q}}

\numberwithin{equation}{section}

\begin{document}

\title [Imaginary bicyclic biquadratic number fields]{Class numbers of Imaginary bicyclic biquadratic number fields}

\author[A. Jakhar]{Anuj Jakhar\, \orcidlink{0009-0007-5951-2261}}
\address[Jakhar]{Department of Mathematics, Indian Institute of Technology Madras, Chennai, Tamil Nadu, India-600036}
\email{Anujjakhar@iitm.ac.in}

\author[R.~Kalwaniya]{Ravi Kalwaniya\, \orcidlink{0009-0008-6964-5276}}
\address[Kalwaniya]{Indian Institute Of Technology, Chennai, Tamil Nadu 600036, India}
\email{ravikalwaniya3@gmail.com}

\author[M. K. Ram]{Mahesh Kumar Ram\, \orcidlink{0000-0002-5692-5761}}
\address[Ram]{Department of Mathematics, Indian Institute of Technology Madras, Chennai, Tamil Nadu, India-600036}
\email{maheshkumarram621@gmail.com}



\maketitle
 
\begin{abstract}
For any fixed positive integer 
$n$, we provide a method to compute all imaginary bicyclic biquadratic number fields with class number $n$, along with their class group structures, using the list of all imaginary quadratic number fields whose class numbers divide $2n$. We apply this method to list all imaginary bicyclic biquadratic number fields with class numbers $4$, $6$ and $7$. We also present the class group structure of each subfield of these fields.
\end{abstract}

\section{Introduction}

\subsection{Motivation and historical context}
Imaginary bicyclic biquadratic fields are degree $4$ extensions of $\Q$ of the form $\Q(\sqrt{-x},\sqrt{y})$ or $\Q(\sqrt{-x},\sqrt{-y})$, where $x$ and $y$ are square-free positive integers. These fields have attracted interest due to their rich arithmetic properties and connections to the theory of quadratic and abelian extensions.

In \cite{BP}, Brown et al. listed all imaginary bicyclic biquadratic fields having class number $1$, while Buell et al. \cite{BW} provided a complete list of such fields with class number 2. Conner et al. \cite{PJ} investigated forms of imaginary bicyclic biquadratic fields with odd class numbers. Using this result, Jung et al. \cite{SS} listed all such fields having class number $3$,  and for class number 
$5$, a complete list was given by Basilla et al. \cite{BF24} by using similar arguments as in \cite{SS}.

Imaginary quadratic fields are well understood up to class number $100$. In this article, we use those with class numbers up to 
$14$ to list imaginary bicyclic biquadratic fields with class numbers $4,\; 6$ and $7$. For any integer $n\geq 1$, if all the imaginary quadratic fields with class numbers up to $2n$ are known, then our method is effective in determining all imaginary bicyclic biquadratic fields with class number $n$. 

The study of class numbers of imaginary abelian number fields is long-standing. For instance, Yamamura~\cite{Y} proved that there exist exactly \(172\) imaginary abelian number fields with class number one. Similarly, Setzer~\cite{B} determined all imaginary cyclic quartic fields with class number \(1\). For further contributions in this area, we refer the reader to~\cite{AB66, AB67, B67, B71, GSS, KH89, MR23, MR25, HS67, HS75, KU72}.

\subsection{Main Results}

Our principal results are as follows.

\begin{thm}\label{T4}
There are exactly $408$ imaginary bicyclic biquadratic number fields with class number $4$. Furthermore, among these, only $296$ fields have cyclic class groups.
\end{thm}

\begin{thm}\label{T6}
There are exactly $552$ imaginary bicyclic biquadratic number fields with class number $6$. 
\end{thm}

\begin{thm}\label{T7}
There are exactly $277$ imaginary bicyclic biquadratic number fields with class number $7$.
\end{thm}
We remark that if $K =\Q(\sqrt{-x}, \sqrt{-y})$ is an imaginary bicyclic biquadratic number field,  where $x$ and $y$ are square-free positive integers, then there exists a square-free positive integer $d$ such that
\[
K = \mathbb{Q}(\sqrt{-x}, \sqrt{d}) = \mathbb{Q}(\sqrt{-y}, \sqrt{d}),
\]
and conversely, any such representation determines the same field. Therefore, when listing number fields \(K\) with a given class number, it suffices to consider only one canonical form. In this article, we choose to represent \(K\) as 
\[
K = \mathbb{Q}(\sqrt{-x}, \sqrt{-y}),
\] 
while noting that our method applies equally to any equivalent form.
\\

The following notations will be used throughout this article. Let \( x \) and \( y \) denote square-free positive integers, and let \( d_{xy} \) denote the square-free part of the product \( xy \). We consider imaginary bicyclic biquadratic number fields of the form
\[
K = \mathbb{Q}(\sqrt{-x}, \sqrt{-y}),
\]
whose three distinct quadratic subfields are given by
\[
K_1 = \mathbb{Q}(\sqrt{-x}), \quad K_2 = \mathbb{Q}(\sqrt{-y}), \quad \text{and} \quad K_3 = \mathbb{Q}(\sqrt{d_{xy}}),
\]
where \( K_3 \) is the unique real quadratic subfield of \( K \).

Throughout, $p_1, p_2, p_3 \equiv 1 \pmod{4}$ and $q_1, q_2, q_3 \equiv 3 \pmod{4}$ will represent distinct primes as needed in subsequent constructions. For any number field $L$, let $\mathcal{O}_L$ denote its ring of integers, $C\ell(L)$ its class group, $h_L$ its class number, and $D_L$ its discriminant. 

\subsection{Organization of the Paper}

Including the introduction, this article is structured into ten sections. Section~2 is devoted to preliminaries. The proofs of Theorems~\ref{T4},~\ref{T6}, and~\ref{T7} are presented in Sections~3,~4, and~5, respectively. Section~6 offers concluding remarks. Sections~7,~8, and~9 provide complete lists of imaginary bicyclic biquadratic number fields with class numbers $4$, $6$, and $7$, respectively. The final section is an appendix that contains a table of imaginary quadratic number fields with small class numbers used in our computations.

\medskip


\section{Preliminaries}

In this section, we recall some results that form the backbone of our arguments throughout the paper. We begin with the following straightforward result.

\begin{prop}\label{P21}
Let $K$ be an imaginary bicyclic biquadratic field, and let $F$ be its real quadratic subfield. Then $h_F$ divides $h_K$.
\end{prop}

The next result, taken from \cite{FL94}, plays a crucial role in our proofs.

\begin{thm}\label{2.1}
Let \( K \) be an imaginary bicyclic biquadratic field with quadratic subfields \( K_1 \), \( K_2 \), and \( K_3 \). Then
\[
h_K = \frac{q(K)\,h_{K_1}\,h_{K_2}\,h_{K_3}}{2},
\]
where $q(K) = [E_K : E_{K_1}E_{K_2}E_{K_3}]$, with $E_K$ and $E_{K_i}$ denoting the unit groups of $\mathcal{O}_K$ and $\mathcal{O}_{K_i}$, respectively. Furthermore,
\[
h_K \geq \frac{h_{K_1} h_{K_2} h_{K_3}}{2},
\]
since \( q(K) \in \{1, 2\} \).
\end{thm}

From Gauss's genus theory, we record the following classical result on quadratic fields.

\begin{thm}\label{2.2}
Let $F$ be a quadratic field with discriminant $D_F$. Let $t$ denote the number of distinct prime divisors of $D_F$, and let $r_2(F)$ denote the $2$-rank of $C\ell(F)$. Then
\[
r_2(F) =
\begin{cases}
t - 2, & \text{if $F$ is real and some prime $p \equiv 3 \pmod{4}$ divides $D_F$,} \\
t - 1, & \text{otherwise.}
\end{cases}
\]
\end{thm}

The following corollary is an easy consequence of Theorem~\ref{2.2}.

\begin{cor}\label{C24}
Let $F = \Q(\sqrt{-z})$ be an imaginary quadratic field, where $z$ is a square-free positive integer. Then for any non-negative integer $s$, we have $2^s \mid h_F$ if one of the following holds:
\begin{itemize}
    \item $z$ has at least $s+1$ distinct prime divisors;
    \item $z$ has exactly $s$ distinct prime divisors and $z \equiv 1 \pmod{4}$.
\end{itemize}
\end{cor}

Analogously to Corollary~\ref{C24}, we obtain the following result from Theorem~\ref{2.2} for real quadratic fields.
\begin{cor}\label{C25}
Let $F = \Q(\sqrt{z})$ be a real quadratic number field, where $z$ is a square-free positive integer. Then, for any non-negative integer $s$, we have $2^s \mid h_F$ if one of the following conditions holds:
\begin{itemize}
    \item $z$ has at least $s+2$ distinct prime divisors;
    \item $z$ has exactly $s+1$ distinct prime divisors and $z \equiv 3 \pmod{4}$;
    \item $z$ has exactly $s+1$ distinct prime divisors and none of its prime divisors is congruent to $3 \pmod{4}$.
\end{itemize}
\end{cor}

\section{Proof of Theorem \ref{T4}}
We now turn to the proof of Theorem~\ref{T4}. Our goal is to determine all imaginary bicyclic biquadratic fields $K = \mathbb{Q}(\sqrt{-x},\, \sqrt{-y})$ with class number $h_K = 4$. Following the notation established in Section~1, we begin this section by proving the following result.
\begin{prop} \label{P31} Let \( K = \mathbb{Q}(\sqrt{-x}, \sqrt{-y}) \) be an imaginary bicyclic biquadratic field. If \( d_{xy} \) has at least four distinct prime factors, then \( h_K \neq 4 \).
\end{prop}
\begin{proof} Let \( K_1 = \mathbb{Q}(\sqrt{-x}) \), \( K_2 = \mathbb{Q}(\sqrt{-y}) \), and \( K_3 = \mathbb{Q}(\sqrt{d_{xy}}) \). By Corollary~\ref{C25}, if \( d_{xy} \) has at least five distinct prime factors, then \( 8 \mid h_{K_3} \). Since \( K_3 \) is a real quadratic field, Proposition~\ref{P21} implies \( 8 \mid h_K \), and thus \( h_K \neq 4 \).

Now suppose \( d_{xy} = pqrt \), where \( p, q, r, t \) are distinct primes. Then by Corollary~\ref{C25}, \( 4 \mid h_{K_3} \). Without loss of generality, the possibilities for $(x, y)$ are as follows:

\noindent
\textbf{Case 1:} \( x = v \) and \( y = pqrtv \), where \( v \) is a square-free positive integer. \\
Then \( K_2 = \mathbb{Q}(\sqrt{-pqrtv}) \), which has at least four distinct prime divisors. Hence, by Corollary~\ref{C24}, \( 8 \mid h_{K_2} \), and since \( 4 \mid h_{K_3} \), Theorem~\ref{2.1} gives \( h_K > 4 \).

\noindent
\textbf{Case 2:} \( x = pv \) and \( y = qrtv \), with square-free positive integer \( v \). \\
Then \( K_2 = \mathbb{Q}(\sqrt{-qrtv}) \) has at least three distinct prime divisors, and hence \( 4 \mid h_{K_2} \). Again, using Theorem~\ref{2.1}, we conclude \( h_K > 4 \).

\noindent
\textbf{Case 3:} \( x = pqv \) and \( y = rtv \), where \( v \) is a square-free positive integer. \\
In this case, again using Corollary~\ref{C24}, \( 2 \mid h_{K_1} \) and \( 2 \mid h_{K_2} \). Therefore, from Theorem~\ref{2.1}, we have \( h_K \neq 4 \).

\vspace{0.2cm}
In all cases, we conclude that if \( d_{xy} \) has at least four distinct prime factors, then \( h_K \neq 4 \). This completes the proof.
\end{proof}
From Proposition~\ref{P31}, we deduce that if \( h_K = 4 \), then $d_{xy}$ has at most three prime factors. Since \( K_3 = \mathbb{Q}(\sqrt{d_{xy}}) \) is a real quadratic field, we must have \( d_{xy} \neq 1 \). We now analyze each possible case for \( d_{xy} \) individually.
We begin with the situation where \( d_{xy} \) is a prime number.

 \begin{prop} \label{P33} Let \( K = \mathbb{Q}(\sqrt{-x}, \sqrt{-y}) \) be an imaginary bicyclic biquadratic number field such that \( d_{xy} \) is a prime number. Then \( h_K = 4 \) if and only if \( K \) is one of the fields listed in the following table:
 \renewcommand{\arraystretch}{1.2}
 \begin{center}
$$
\begin{array}{|c|l|c|}
\hline
\# & \text{Field } K & [h_{K_1},\ h_{K_2},\ h_{K_3},\ q(K)] \\
\hline
1 & \mathbb{Q}(\sqrt{-p_1p_2q_1}, \sqrt{-p_2q_1}) & [4,\ 2,\ 1,\ 1]\\
\hline
2 & \mathbb{Q}(\sqrt{-2p_1p_2}, \sqrt{-2p_2}) & [4,\ 2,\ 1,\ 1]\\
\hline
3 & \mathbb{Q}(\sqrt{-2p_1q_1}, \sqrt{-2q_1}) & [4,\ 2,\ 1,\ 1]\\
\hline
4 & \mathbb{Q}(\sqrt{-p_1p_2}, \sqrt{-p_1}) & [4,\ 2,\ 1,\ 1]\\
\hline
5 & \mathbb{Q}(\sqrt{-p_1q_1}, \sqrt{-q_1}) & [4,\ 1,\ 1,\ 2],\  [8,\ 1,\ 1,\ 1]\\
\hline
6 & \mathbb{Q}(\sqrt{-2p_1}, \sqrt{-2}) & [4,\ 1,\ 1,\ 2],\  [8,\ 1,\ 1,\ 1]\\
\hline

7 & \mathbb{Q}(\sqrt{-p_1}, \sqrt{-1}) & [4,\ 1,\ 1,\ 2],\  [8,\ 1,\ 1,\ 1]\\
\hline
8 & \mathbb{Q}(\sqrt{-2p_1q_1}, \sqrt{-2p_1}) & [4,\ 2,\ 1,\ 1]\\
\hline
9 & \mathbb{Q}(\sqrt{-2q_1q_2}, \sqrt{-2q_2}) & [4,\ 2,\ 1,\ 1]\\
\hline

10 & \mathbb{Q}(\sqrt{-p_1q_1}, \sqrt{-p_1}) & [2,\ 2,\ 1,\ 2],\ [4,\ 2,\ 1,\ 1],\ [2,\ 4,\ 1,\ 1]\\
\hline
11 & \mathbb{Q}(\sqrt{-q_1q_2}, \sqrt{-q_1}) & [4,\ 1,\ 1,\ 2],\  [8,\ 1,\ 1,\ 1]\\
\hline
12 & \mathbb{Q}(\sqrt{-2q_1}, \sqrt{-2}) & [4,\ 1,\ 1,\ 2],\  [8,\ 1,\ 1,\ 1]\\
\hline

13 & \mathbb{Q}(\sqrt{-2p_1q_1}, \sqrt{-p_1q_1}) & [4,\ 2,\ 1,\ 1]\\
\hline
14 & \mathbb{Q}(\sqrt{-2p_1}, \sqrt{-p_1}) & [2,\ 2,\ 1,\ 2],\  [4,\ 2,\ 1,\ 1],\  [2,\ 4,\ 1,\ 1]\\
\hline
15 & \mathbb{Q}(\sqrt{-2q_1}, \sqrt{-q_1}) & [4,\ 1,\ 1,\ 2],\  [8,\ 1,\ 1,\ 1]\\
\hline

\end{array}
$$
\captionof{table}{}
\label{tab1}
\end{center}
\end{prop}
\begin{proof} Let $h_K = 4$ and $d_{xy} = p\; \text{or}\; 2$,  where $p$ is an odd prime. Then the real quadratic subfield $K_3 = \Q(\sqrt{d_{xy}})$ has odd class number by virtue of Theorem~\ref{2.2}. In this setting, it follows that either $(x, y) = (pv, v)$ or $(x, y) = (2v, v)$ for some square-free positive integer $v$. 
 As $h_K = 4$, Corollary \ref{C24} and Theorem \ref{2.1} imply that $v$ must have at most two distinct prime factors. We now split the analysis into cases according to $d_{xy}$.\\

\noindent
\textbf{Case (i):} Let $d_{xy} = p \equiv 1 \pmod{4}$, with $x = pv$ and $y = v$. \\
\noindent
\textbf{ I:} Let $v = qr$ or $v =2r$, where $q$ and $r$ are distinct odd primes. For $v = qr$, we have either $v \equiv 1 \pmod 4$ or $v \equiv 3 \pmod 4$. 

\noindent
\textbf{I(A):} Consider $v = qr \equiv 1 \pmod 4 $. Then $x = pqr \equiv 1 \pmod 4$ and $y =qr \equiv 1 \pmod 4 $. Using Corollary \ref{C24}, it follows that $8 \mid h_{K_1}$ and $4 \mid h_{K_2}$. Applying Theorem \ref{2.1}, we conclude that $h_K \neq 4$. Thus, this case does not yield any fields with class number $4$.\\
\textbf{I(B):} Consider $v = qr \equiv 3 \pmod 4$. Without loss of generality, we may assume that  $q \equiv 1 $ and $r \equiv 3 \pmod 4$. By Corollary \ref{C24}, we have $4 \mid h_{K_1}$ and $2 \mid h_{K_2}$. Applying Theorem~\ref{2.1}, we see that $h_K = 4$ can occur only when $h_{K_1} = 4$, $h_{K_2} = 2$, $h_{K_3} = 1$, and $q(K) = 1$, i.e., $[h_{K_1}, h_{K_2}, h_{K_3}, q(K)] = [4, 2, 1, 1].$
By relabeling $p$, $q$, and $r$ as $p_1$, $p_2$, and $q_1$, respectively, we obtain the first row of Table~\ref{tab1}.

\noindent
\textbf{I(C):} Suppose $v = 2r$ for an odd prime $r$. Then $x = 2 p r$ and $y = 2 r$. Applying Corollary~\ref{C24}, we again obtain $4 \mid h_{K_1}$ and $2 \mid h_{K_2}$ for any odd prime $r$. Using Theorem~\ref{2.1}, we see that $h_K = 4$ if and only if $h_{K_1} = 4$, $h_{K_2} = 2$, $h_{K_3} = 1$, and $q(K) = 1$. Relabeling $p$ and $r$ as $p_1$ and $p_2$ when $r \equiv 1 \pmod{4}$, or as $p_1$ and $q_1$ when $r \equiv 3 \pmod{4}$, yields the second and third rows of Table~\ref{tab1}.\\

\noindent
\textbf{II:} Now consider \( v = q \) or \( v = 2 \), where \( q \) is an odd prime.\\
\textbf{II(A):} Let $v = q \equiv 1 \pmod{4}$. Then $x = p q \equiv 1 \pmod{4}$ and $y = q \equiv 1 \pmod{4}$. By Corollary~\ref{C24}, it follows that $4 \mid h_{K_1}$ and $2 \mid h_{K_2}$. Applying Theorem~\ref{2.1}, we find that $h_{K_1} = 4$, $h_{K_2} = 2$, $h_{K_3} = 1$, and $q(K) = 1$. Relabeling $q$ and $p$ as $p_1$ and $p_2$, respectively, yields the fourth row of Table~\ref{tab1}.\\
\textbf{II(B):} Consider $v = q \equiv 3 \pmod{4}$. Then $x = p q \equiv 3 \pmod{4}$ and $y = q \equiv 3 \pmod{4}$. By Corollary~\ref{C24}, we have $2 \mid h_{K_1}$ and $h_{K_2}$ is odd. Applying Theorem~\ref{2.1}, the possible values for the tuple $\big[h_{K_1}, h_{K_2}, h_{K_3}, q(K)\big]$ are $\{[4,\, 1,\, 1,\, 2],\, [8,\, 1,\, 1,\, 1]\}$. Relabeling $p$ and $q$ as $p_1$ and $q_1$, respectively, gives the fifth row of Table~\ref{tab1}.\\
\textbf{II(C):} Consider the case where $v = 2$. Then $x = 2p$ and $y = 2$. In this situation, Corollary~\ref{C24} together with Theorem~\ref{2.1} shows that the possible values for the tuple $\big[h_{K_1},\, h_{K_2},\, h_{K_3},\, q(K)\big]$ are $\{[4,\, 1,\, 1,\, 2],\, [8,\, 1,\, 1,\, 1]\}$. Relabeling $p$ with $p_1$ provides the sixth row of Table~\ref{tab1}.\\

\noindent \textbf{III:} Consider $v = 1$. Then we have $K = \Q(\sqrt{-p},\, \sqrt{-1})$. From Corollary~\ref{C24}, we have $2 \mid h_{K_1}$. We know that the class number of $\Q(\sqrt{-1})$ is $1$.  Therefore, applying Theorem~\ref{2.1}, we see that \[\big[h_{K_1},\, h_{K_2},\, h_{K_3},\, q(K)\big] \in \{[4,\, 1,\, 1,\, 2], [8,\, 1,\, 1,\, 1]\}.\] Setting $p = p_1$ yields the seventh row of Table~\ref{tab1}.\\

\noindent
\textbf{Case (ii):} Assume that $d_{xy} = p \equiv 3 \pmod 4$, and let $x = pv $ and $y= v$.\\
\textbf{I:} Let $v = qr$ or $v = 2r$, where $q$ and $r$ are  distinct odd primes. The value of $v = qr$ can satisfy either $v \equiv 1 \pmod 4$ or $v \equiv 3 \pmod 4$.\\
\textbf{I(A):} Assume that $v \equiv 1 \pmod 4$. We see that $x$ has three distinct prime factors, and $y \equiv 1 \pmod 4$ has two distinct prime factors. Therefore, Corollary \ref{C24} implies that $4\mid h_{K_1}$ and $4 \mid h_{K_2}$. Consequently, by Theorem \ref{2.1}, we conclude that $h_K \neq 4$. Thus, we discard this case.\\
\textbf{I(B):} For $v \equiv 3 \pmod 4$, we have $x \equiv 1 \pmod 4$ and $y \equiv \ 3 \pmod 4$. Thus, applying Corollary \ref{C24} and Theorem \ref{2.1} again, we obtain $h_K \neq 4$. Hence, this case is not possible.\\
\textbf{I(C):} Consider $v =2r$. Then, $x =2pr$ and $y =2r$. From Corollary \ref{C24}, it follows that  $4 \mid h_{K_1}$ and $2 \mid h_{K_2}$ for any odd prime $r$. Using this in Theorem \ref{2.1} to obtain $h_K =4$, we find that $h_{K_1}=4, h_{K_2}=2, h_{K_3}=1$ and $q(K) =1$ for any odd prime $r$. Replacing $p$ and $r$ with $q_1$ and $p_1$ when $r \equiv 1 \pmod 4$, and with $q_1$ and $q_2$ when $r \equiv 3 \pmod 4$ respectively,  yields the $8$th  and $9$th rows of Table \ref{tab1}.\\

\noindent
\textbf{II:} Consider $v = q$ or $v =2$, where $q$ is an odd prime.\\
\textbf{II(A):} Consider $v = q \equiv 1 \pmod 4$. Then, $x =pq \equiv 3 \pmod 4$ and $y =q \equiv 1 \pmod 4$. From Corollary \ref{C24}, it follows that $2 \mid h_{K_1}$ and $2 \mid h_{K_2}$. Thus,  Theorem \ref{2.1} gives that $$[h_{K_1}, h_{K_2}, h_{K_3}, q(K)] \in \{[2,\ 2,\ 1,\ 2],\ [4,\ 2,\ 1,\ 1],\ [2,\ 4,\ 1,\ 1]\}.$$ Using $p_1$ and $q_1$ in place of $q$ and $p$, respectively, we obtain the $10$th row of  Table \ref{tab1}.\\
\textbf{II(B):} Consider $v = q \equiv 3 \pmod 4$. Then, we have $x =pq \equiv 1 \pmod 4$ and $y =q \equiv 3 \pmod 4$. From Corollary \ref{C24}, it follows that $4 \mid h_{K_1}$ and $h_{K_2}$ is odd. Applying Theorem \ref{2.1}, we obtain that the possible values for the tuple  $[h_{K_1}, h_{K_2}, h_{K_3}, q(K)]$ are $[4, 1, 1, 2]$ and $  [8, 1, 1, 1].$ Identifying 
$p$ and $q$ with $q_2$ and $q_1$, respectively, confirms the $11$th row of Table \ref{tab1}.\\
\textbf{II(C):} Let $v = 2$. Then, $x = 2p$ and $y = 2$. In this case, Corollary \ref{C24} and Theorem \ref{2.1} together imply that $$[h_{K_1}, h_{K_2}, h_{K_3}, q(K)] \in \{[4, 1, 1, 2], [8, 1, 1, 1]\}.$$ Replacing $p$ by $q_1$, we obtain the $12$th row of Table \ref{tab1}.\\

\noindent
\textbf{III:} Consider $v = 1$. Then, $x = p$ and $y = 1$. By Corollary \ref{C24}, $h_{K_1}$ is odd, and it is known that $h_{K_2} =1$. Using these facts in Theorem \ref{2.1} to obtain $h_{K} = 4$, we get that $q(K) >2$, which is not possible. Thus, we exclude this case.\\

\noindent
\textbf{Case (iii):} Let $d_{xy} = 2$. Consider the case  $x =2v$ and $y = v$. In this case, $h_{K_3} = 1$.\\
\textbf{I:} Suppose $v=qr$, where $q$ and $r$ are distinct odd primes. \\
\textbf{I(A):} Consider $v = qr \equiv 1 \pmod 4$. Then, $x = 2qr \equiv 2 \pmod 4 $ and $y = qr\equiv 1 \pmod 4$. From Corollary \ref{C24}, it follows that $4$ divides both $h_{K_1}$ and $h_{K_2}$, which is not possible by Theorem \ref{2.1}, as we require $h_K = 4$.\\
\textbf{I(B):} Consider $v = qr \equiv 3 \pmod 4$. Then, $x = 2qr \equiv 2 \pmod 4 $ and $y = qr \equiv 3 \pmod 4$. From Corollary \ref{C24}, we find that $4$ divides  $h_{K_1}$ and $2$ divides $h_{K_2}$. Thus, from Theorem \ref{2.1} for $h_K =4$, we have $[h_{K_1}, h_{K_2}, h_{K_3}, q(K)] = [4, 2, 1, 1]$. Without loss of generality, we may assume that $q \equiv 1 \pmod 4$ and $r \equiv 3 \pmod 4$. Replacing $q$ and $r$ with $p_1$ and $q_1$, respectively, gives the $13$th row of Table \ref{tab1}.\\

\noindent
\textbf{II:} Suppose $v=q$, where $q$ is an odd prime. Then, $x = 2q$ and $y = q$.\\
\noindent
\textbf{II(A):} Let $q \equiv 1 \pmod 4$. Using Corollary \ref{C24} and Theorem \ref{2.1}, we obtain that 
$$[h_{K_1}, h_{K_2}, h_{K_3}, q(K)] \in \{[2,\ 2,\ 1,\ 2],\  [4,\ 2,\ 1,\ 1],\  [2,\ 4,\ 1,\ 1]\}.$$ Using $p_1$ for $q$, we conclude the $14$th row of Table \ref{tab1}.\\
\textbf{II(B):} Let $q \equiv 3 \pmod 4$. Using a similar argument as in the previous case yields the  $15$th row of  Table \ref{tab1}.\\

\noindent
\textbf{III:} Suppose $v=1$. Then, we have $K =\Q (\sqrt{-2}, \sqrt{-1})$. Here, $h_{K_i} = 1$ for $i = 1, 2, 3.$ Thus, to obtain $h_{K} = 4$, Theorem \ref{2.1} implies that $q(K) > 2$, which is not possible. Thus, we discard this case.\\

Conversely, using Theorem \ref{2.1}, it is clear that each field $K$ listed in Table \ref{tab1} has class number $4$. This completes the proof. 
\end{proof}

\begin{prop} \label{P34} Let $K=  \Q (\sqrt{-x}, \sqrt{-y})$ such that $d_{xy}$ has exactly two distinct prime factors. Then, $h_K = 4$ if and only if $K$ is one of the fields listed in the following table:
\begin{center}   
\begin{longtable*}[h!]{|c|l|c|}
\hline
\# & \text{Field } K & $[h_{K_1},\ h_{K_2},\ h_{K_3},\ q(K)]$ \\
\hline
$1$ & $\mathbb{Q}(\sqrt{-p_1p_2q_1}, \sqrt{-q_1})$ & $[4,\ 1,\ 2,\ 1]$\\
\hline
$2$ & $\mathbb{Q}(\sqrt{-2p_1p_2}, \sqrt{-2})$ & $[4,\ 1,\ 2,\ 1]$\\
\hline
$3$ & $\mathbb{Q}(\sqrt{-p_1p_2}, \sqrt{-1})$ & $[4,\ 1,\ 2,\ 1]$\\
\hline
$4$ & $\mathbb{Q}(\sqrt{-p_1q_1}, \sqrt{-p_2q_1})$ & $[2,\ 2,\ 2,\ 1]$\\
\hline

$5$ & $\mathbb{Q}(\sqrt{-2p_1}, \sqrt{-2p_2})$ & $[2,\ 2,\ 2,\ 1]$\\
\hline

$6$ & $\mathbb{Q}(\sqrt{-p_1}, \sqrt{-p_2})$ & $[2,\ 2,\ 2,\ 1]$\\
\hline
$7$ & $\mathbb{Q}(\sqrt{-q_1q_2q_3}, \sqrt{-q_1})$ & $[4,\ 1,\ 1,\ 2],\ [8,\ 1,\ 1,\ 1]$\\
\hline
$8$ & $\mathbb{Q}(\sqrt{-2q_1q_2}, \sqrt{-2})$ & $[4,\ 1,\ 1,\ 2],\ [8,\ 1,\ 1,\ 1]$\\
\hline
$9$ & $\mathbb{Q}(\sqrt{-q_1q_2}, \sqrt{-1})$ & $[4,\ 1,\ 1,\ 2],\ [8,\ 1,\ 1,\ 1]$\\
\hline
$10$ & $\mathbb{Q}(\sqrt{-p_1q_1}, \sqrt{-p_1q_2})$ & $[2,\ 2,\ 1,\ 2],\ [2,\ 4,\ 1,\ 1]$ \\
\hline

$11$ & $\mathbb{Q}(\sqrt{-2q_1}, \sqrt{-2q_2})$ & $[2,\ 2,\ 1,\ 2],\ [2,\ 4,\ 1,\ 1]$ \\
\hline

$12$ & $\mathbb{Q}(\sqrt{-2p_1q_1}, \sqrt{-2})$ & $[4,\ 1,\ 2,\ 1]$\\
\hline
$13$ & $\mathbb{Q}(\sqrt{-p_1q_1}, \sqrt{-1})$ & $[2,\ 1,\ 2,\ 2],\ [4,\ 1,\ 2,\ 1],\ [2,\ 1,\ 4,\ 1]$\\
\hline
$14$ & $\mathbb{Q}(\sqrt{-2p_1}, \sqrt{-2q_1})$ & $[2,\ 2,\ 2,\ 1]$\\
\hline

$15$ & $\mathbb{Q}(\sqrt{-p_1}, \sqrt{-q_1})$ & $[2,\ 1,\ 2,\ 2],\ [4,\ 1,\ 2,\ 1],\ [2,\ 1,\ 4,\ 1]$\\
\hline
$16$ & $\mathbb{Q}(\sqrt{-2p_1q_1}, \sqrt{-q_1})$ &  $[4,\ 1,\ 2\ 1]$\\
\hline
$17$ & $\mathbb{Q}(\sqrt{-2p_1}, \sqrt{-1})$ &  $[2,\ 1,\ 2,\ 2],\ [4,\ 1,\ 2,\ 1],\ [2,\ 1,\ 4,\ 1]$\\
\hline
$18$ & $\mathbb{Q}(\sqrt{-2q_1}, \sqrt{-p_1q_1})$ &  $[2,\ 2,\ 2,\ 1]$\\
\hline

$19$ & $\mathbb{Q}(\sqrt{-p_1}, \sqrt{-2})$ & $[2,\ 1,\ 2,\ 2],\ [4,\ 1,\ 2,\ 1],\ [2,\ 1,\ 4,\ 1]$\\
\hline
$20$ & $\mathbb{Q}(\sqrt{-2p_1q_1}, \sqrt{-p_1})$ & $[4,\ 2,\ 1,\ 1]$\\
\hline
$21$ & $\mathbb{Q}(\sqrt{-2q_1q_2}, \sqrt{-q_1})$ & $[4,\ 1,\ 1,\ 2],\ [8,\ 1,\ 1,\ 1]$\\
\hline
$22$ & $\mathbb{Q}(\sqrt{-2q_1}, \sqrt{-1})$ & $[4,\ 1,\ 1,\ 2],\ [8,\ 1,\ 1,\ 1]$\\
\hline
$23$ & $\mathbb{Q}(\sqrt{-p_1q_1}, \sqrt{-2 p_1})$ & $[2,\ 2,\ 1,\ 2],\ [4,\ 2,\ 1,\ 1],\ [2,\ 4,\ 1,\ 1]$\\
\hline

$24$ & $\mathbb{Q}(\sqrt{-q_1q_2}, \sqrt{-2q_1})$ & $[4,\ 2,\ 1,\ 1]$\\
\hline

\end{longtable*}
\captionof{table}{}
\label{tab2}
\end{center}
\end{prop}
\begin{proof}
Recall that $K = \mathbb{Q}(\sqrt{-x}, \sqrt{-y})$, with $K_1 = \mathbb{Q}(\sqrt{-x})$, $K_2 = \mathbb{Q}(\sqrt{-y})$, and $K_3 = \mathbb{Q}(\sqrt{d_{xy}})$. Clearly, $x \neq y$. Let $h_K = 4$ and $d_{xy} = pq$ or $d_{xy} = 2p$, where $p$ and $q$ are odd primes. Since $K_1$ and $K_2$ are imaginary number fields, for $d_{xy} \equiv 1, 3 \pmod 4$, without loss of generality, the possible values of $x$ and $y$ are as follows:
\begin{itemize}
    \item  $x = pqv $ and $y= v$,
    \item $x = pv$ and $y = qv$,
     \end{itemize}
where $v$ is some square-free positive integer.
Since $K_1$ and $K_2$ are imaginary number fields,  for $d_{xy} \equiv 2\pmod 4$, we have the following possible pairs $(x, y)$:
\begin{itemize}
\item  $x = 2pv $ and $y= v$,
      \item $x = pv$ and $y = 2v$,
 \end{itemize}
where $v$ is some square-free positive integer. Since $h_K = 4$, from Corollary \ref{C24} and Theorem \ref{2.1}, it is clear that $v$ must have at most one prime factor in all possible pairs $(x, y)$ for the cases $d_{xy} \equiv 1, 3 \pmod 4$ and $d_{xy} \equiv 2\pmod 4$. Thus, $v$ can be an odd prime $r$, or equal to $2$ or $1$. Note that the case $v = 2$ is not possible for $d_{xy} \equiv 2\pmod 4$.\\ \\
\textbf{Case (i):} Consider the case where $d_{xy} = pq \equiv 1 \pmod 4$. This leads to two subcases: $p, q \equiv 1 \pmod 4$; or $p, q \equiv 3 \pmod 4$.\\
\textbf{Subcase (a):} Let $p \equiv 1 \pmod 4$ and $q \equiv 1 \pmod 4$. By Corollary \ref{C25}, we have $2 \mid h_{K_3}$.\\
\textbf{I:} Assume that $x = pqv $ and $y= v$.\\
\textbf{I(A):} Let $v = r$, where $r$ is an odd prime. Corollary \ref{C24} and Theorem \ref{2.1} ensure that $r \not \equiv 1 \pmod 4$. Therefore, assume that $r  \equiv 3 \pmod 4$. Applying Corollary \ref{C24} and Theorem \ref{2.1} again, we have $[h_{K_1},\ h_{K_2},\ h_{K_3},\ q(K)]$ $= [4,\ 1,\ 2,\ 1]$. Thus, replacing $p, q, r$ with $p_1, p_2, q_1$, respectively, we obtain the first row of Table \ref{tab2}.\\
\textbf{I(B):} Let $v = 2$. Then, $x = 2pq$ and $y = 2$. The class number of $\Q(\sqrt{-2})$ is $1$, 
and by Corollary \ref{C24} we have $4 \mid h_{K_1}$. Thus, Theorem \ref{2.1} implies that 
$[h_{K_1},\ h_{K_2},\ h_{K_3},\ q(K)] = [4,\ 1,\ 2,\ 1]$. Now, using $p_1$ and $p_2$ for $p$ and $q$, respectively, we find the $2$nd row of Table \ref{tab2}.\\
\textbf{I(C):} Let $v = 1$. Then, $x = pq$ and $y = 1$. The class number of $\Q(\sqrt{-1})$ is $1$. Since $x$ has two distinct prime factors and $x \equiv 1 \pmod 4$, Corollary \ref{C24} implies that $4 \mid h_{K_1}$. Now, applying Theorem \ref{2.1}, we have 
$[h_{K_1},\ h_{K_2},\ h_{K_3},\ q(K)] = [4,\ 1,\ 2,\ 1]$. Thus, changing  $p$ and $q$ by $p_1$ and $p_2$, respectively, we have the $3$rd row of Table \ref{tab2}.\\

\noindent
\textbf{II:} Assume that $x = pv$ and $y = qv$.\\
\textbf{II(A):} Let $v = r$, where $r$ is an odd prime. Then, $x = pr$ and $y = qr$. From Corollary \ref{C24} and Theorem \ref{2.1}, it follows that $r\not \equiv 1 \pmod 4$. So, let $r\equiv 3 \pmod 4$. Applying Corollary \ref{C24} and Theorem \ref{2.1} again, we have $[h_{K_1},\ h_{K_2},\ h_{K_3},\ q(K)] = [2,\ 2,\ 2,\ 1]$. Thus, replacing $p, q, r$ with $p_1, p_2, q_1$, respectively, we obtain the $4$th row of Table \ref{tab2}.\\
\textbf{II(B):} Let $v = 2$. Then, $x = 2p$ and $y = 2q$. By Corollary \ref{C24}, we have $2 \mid h_{K_1}$ and $2 \mid h_{K_2}$. Thus, Theorem \ref{2.1} guarantees that 
$[h_{K_1},\ h_{K_2},\ h_{K_3},\ q(K)] = [2,\ 2,\ 2,\ 1]$. Now, using $p_1$ and $p_2$ for $p$ and $q$, respectively, we have the $5$th row of Table \ref{tab2}.\\
\textbf{II(C):} Let $v = 1$. Then, $x = p$ and $y = q$. Since $x$ and $y$ are primes and $x, y \equiv 1 \pmod 4$, from Corollary \ref{C24} we find that $2 \mid h_{K_1}$ and $2 \mid h_{K_2}$. Thus, Theorem \ref{2.1} ensures that 
$[h_{K_1},\ h_{K_2},\ h_{K_3},\ q(K)] = [2,\ 2,\ 2,\ 1]$. Hence, changing  $p$ and $q$ by $p_1$ and $p_2$, respectively, we obtain the $6$th row of Table \ref{tab2}.\\

\noindent
\textbf{Subcase (b):} Let $p \equiv 3 \pmod 4$ and $q \equiv 3 \pmod 4$. By Corollary \ref{C25}, we obtain that $h_{K_3}$ is odd. Thus, for $h_K = 4$, Proposition \ref{P21} implies that $h_{K_3} = 1$.\\
\textbf{I:} Assume that $x = pqv $ and $y= v$.\\
\textbf{I(A):} Let $v = r$, where $r$ is an odd prime. Then, $x = pqr $ and $y= r$. Corollary \ref{C24} and Theorem \ref{2.1} confirm that $r \not\equiv 1 \pmod 4$. So, let $r \equiv 3 \pmod 4$. By Corollary \ref{C24}, $4 \mid  h_{K_1}$ and $h_{K_2}$ is odd. Consequently, Theorem \ref{2.1} gives:$$[h_{K_1},\ h_{K_2},\ h_{K_3},\ q(K)] \in \{ [4,\ 1,\ 1,\ 2], [8,\ 1,\ 1,\ 1]\}.$$ The $7$th row of Table \ref{tab2} is thus obtained by replacing $r, q, p$ with $q_1, q_2, q_3$, respectively.\\
\textbf{I(B):} Let $v = 2$. Then, $x = 2pq$ and $y = 2$. The class number of $\Q(\sqrt{-2})$ is $1$. By Corollary \ref{C24}, we see that $4 \mid h_{K_1}$. Applying Theorem \ref{2.1}, we get
$[h_{K_1},\ h_{K_2},\ h_{K_3},\ q(K)] \in \{[4,\ 1,\ 1,\ 2], [8,\ 1,\ 1,\ 1]$\}. Thus, substituting  $q_1$ and $q_2$ for $p$ and $q$, respectively,  yields the $8$th row of Table \ref{tab2}.\\
\textbf{I(C):} Let $v = 1$. Then, $x = pq$ and $y = 1$. The class number of $\Q(\sqrt{-1})$ is $1$. Since $x$ has two distinct prime factors and $x \equiv 1 \pmod 4$, from Corollary \ref{C24} we get that $4 \mid h_{K_1}$. Consequently, Theorem \ref{2.1} gives 
$[h_{K_1},\ h_{K_2},\ h_{K_3},\ q(K)] \in \{[4,\ 1,\ 1,\ 2], [8,\ 1,\ 1,\ 1]\}$. Thus, by setting $p = q_1$ and $q = q_2$, we derive the  $9$th row of Table \ref{tab2}.\\

\noindent
\textbf{II:} Assume that $x = pv$ and $y = qv$.\\
\textbf{II(A):} Let $v = r$, where $r$ is an odd prime. Then, $x = pr$ and $y = qr$. Using Corollary \ref{C24} and Theorem \ref{2.1}, we conclude that $r\not \equiv 3 \pmod 4$. So, let $r\equiv 1 \pmod 4$. Applying Corollary \ref{C24} and Theorem \ref{2.1} again, we have $$[h_{K_1},\ h_{K_2},\ h_{K_3},\ q(K)] \in \{[2,\ 2,\ 1,\ 2], [2,\ 4,\ 1,\ 1], [4,\ 2,\ 1,\ 1]\}.$$  Set $p = q_1, q=q_2$, and $r = p_1$. Note that  we have \begin{align*}
\left\{ \Q(\sqrt{-x}, \sqrt{-y}) :\,  x = p_1q_1,\ y = p_1q_2,\ h_{K_1} = 2,\ h_{K_2} = 4 \right\} & \\ =
 \left\{ \Q(\sqrt{-y}, \sqrt{-x}) :\,  x = p_1q_1,\ y = p_1q_2,\ h_{K_1} = 4,\ h_{K_2} = 2 \right\}.
\end{align*}
Thus, for tuples $[h_{K_1},\ h_{K_2},\ h_{K_3},\ q(K)] = [2,\ 4,\ 1,\ 1]$ and  $[h_{K_1},\ h_{K_2},\ h_{K_3},\ q(K)] = [4,\ 2,\ 1,\ 1]$, we obtain  same list of fields $K$. Hence, we consider only one of the tuples, either $[2,\ 4,\ 1,\ 1]$ or $[4,\ 2,\ 1,\ 1].$ This gives the $10$th row of Table \ref{tab2}.\\
\textbf{II(B):} Let $v = 2$. Then, $x = 2p$ and $y = 2q$. From Corollary \ref{C24}, we have $2 \mid h_{K_1}$ and $2 \mid h_{K_2}$. Thus, Theorem \ref{2.1} gives: 
$$[h_{K_1},\ h_{K_2},\ h_{K_3},\ q(K)] \in \{[2,\ 2,\ 1,\ 2], [2,\ 4,\ 1,\ 1], [4,\ 2,\ 1,\ 1]\}.$$  Note that, using arguments similar to the above case, only one of the tuple pairs $([2,\ 2,\ 1,\ 2], [2,\ 4,\ 1,\ 1])$ or $([2,\ 2,\ 1,\ 2], [4,\ 2,\ 1,\ 1])$ is sufficient. Now, using $q_1$ and $q_2$ for $p$ and $q$, respectively,  we obtain the $11$th row of Table \ref{tab2}.\\
\textbf{II(C):} Let $v = 1$. Then, $x = p$ and $y = q$. Since $x$ and $y$ are primes and $x, y \equiv 3 \pmod 4$,  Corollary \ref{C24} implies that   $h_{K_1}$ and $ h_{K_2}$ are odd. Using this, along with the assumation that $h_K = 4$,  Theorem \ref{2.1} implies that $h_{K_i} = 1$ for $i = 1, 2, 3$ and $q(K) = 8$, which is not possible. Therefore, we exclude this case.\\

\noindent
\textbf{Case (ii):} Consider the case where $d_{xy} = pq \equiv 3 \pmod 4$. This leads to two subcases: $p\equiv 1 \pmod 4$ and $q\equiv 3 \pmod 4$; or $p \equiv 3 \pmod 4$ and $q\equiv 1 \pmod 4$. However, in our case, it is sufficient to consider only one subcase. Therefore, we assume that $p\equiv 1 \pmod 4$ and $q\equiv 3 \pmod 4$. Applying Corollary \ref{C25}, we find that $2 \mid h_{K_3}$.\\
\textbf{I:} Assume that $x = pqv $ and $y= v$.\\
\textbf{I(A):} Let $v = r$, where $r$ is an odd prime. Then, $x =  pqr$ and $y = r$. Corollary \ref{C24} and Theorem \ref{2.1} guarantee that this case is not possible.\\
\textbf{I(B):} Let $v = 2$. Then, $x =  2pq$ and $y = 2$. We have $h_{K_2} = 1$, and by Corollary \ref{C24} we see that $4 \mid h_{K_1}$. Applying Theorem \ref{2.1}, we obtain that $[h_{K_1},\ h_{K_2},\ h_{K_3},\ q(K)] = [4,\ 1,\ 2,\ 1]$. The $12$th row of Table \ref{tab2} is thus obtained by replacing $p$ and $q$ with $p_1$ and $q_1$, respectively.\\
\textbf{I(C):} Consider $v = 1$. Then,  $x =  pq$ and $y = 1$. By Corollary \ref{C24}, $2 \mid h_{K_1}$, and it is easy to see that $h_{K_2} = 1$. Therefore, Theorem \ref{2.1} gives: $$[h_{K_1},\ h_{K_2},\ h_{K_3},\ q(K)] \in \{[2,\ 1,\ 2,\ 2], [4,\ 1,\ 2,\ 1], [2,\ 1,\ 4,\ 1]\}.$$ Thus, replacing $p$ and $q$ with $p_1$ and $q_1$, respectively, yields the $13$th row of Table \ref{tab2}.\\

\noindent
\textbf{II:} Assume that $x = pv$ and $y = qv$.\\
\textbf{II(A):} Let $v = r$, where $r$ is an odd prime. Then, $x = pr$ and $y = qr$. In this case, for odd prime $r$, Corollary \ref{C24} and Theorem \ref{2.1} imply that $h_K \neq 4$.\\
\textbf{II(B):} Suppose $v = 2$, then $x = 2p$ and $y = 2q$. From Corollary \ref{C24}, $2 \mid h_{K_1}$ and $2 \mid h_{K_2}$. From this and $2 \mid h_{K_3}$, Theorem \ref{2.1} gives that $[h_{K_1},\ h_{K_2},\ h_{K_3},\ q(K)] = [2,\ 2,\ 2,\ 1]$. Thus, by setting $ p_1 = p$ and $q_1 = q $, we obtain the $14$th row of Table \ref{tab2}.\\
\textbf{II(C):} Suppose $v = 1$, then $x = p$ and $y = q$. Using Corollary \ref{C24}, we see that $2 \mid h_{K_1}$  and $h_{K_2}$ is odd. Therefore, Theorem \ref{2.1} implies that $$[h_{K_1},\ h_{K_2},\ h_{K_3},\ q(K)] \in \{[2,\ 1,\ 2,\ 2], [4,\ 1,\ 2,\ 1], [2,\ 1,\ 4,\ 1]\}.$$ Thus, replacing $p$ and $q$ with $p_1$ and $q_1$, respectively, yields the $15$th row of Table \ref{tab2}.

\noindent
\textbf{Case (iii):} Consider  $d_{xy} = 2p \equiv 2 \pmod 4$. This gives rise to two subcases: $p\equiv 1 \pmod 4$ and $p\equiv 3 \pmod 4$.\\
\textbf{Subcase (a):} Assume that $p \equiv 1 \pmod 4$. Applying Corollary \ref{C25}, we have $2 \mid h_{K_3}$.\\
\textbf{I:} Assume that $x = 2pv $ and $y= v$.\\
\textbf{I(A):} Let $v = r$, where $r$ is an odd prime. Then, $x =  2pr$ and $y = r$. Using Corollary \ref{C24} and Theorem \ref{2.1}, we conclude that $r \not\equiv 1 \pmod 4$. So, let $r \equiv 3 \pmod 4$. Applying Corollary \ref{C24} and Theorem \ref{2.1} again, we have $[h_{K_1},\ h_{K_2},\ h_{K_3},\ q(K)] = [4,\ 1,\ 2,\ 1]$. This establishes the $16$th row of Table \ref{tab2} by substituting $p_1$ and $q_1$ for $p$ and $r$, respectively.\\
\\
\textbf{I(B):} Suppose $v= 1$, then $x =  2p$ and $y = 1$. From this case, we obtain the $17$th row of Table \ref{tab2} by substituting $p_1$ for $p$. \\
\textbf{II:}  Assume that $x = pv$ and $y = 2v$.\\
\textbf{II(A):} Let $v = r$, where $r$ is an odd prime. Then, $x =  pr$ and $y = 2r$. Corollary \ref{C24} and Theorem \ref{2.1} ensure that $r \not \equiv 1 \pmod 4$. So, assume $r \equiv 3 \pmod 4$. Applying Corollary \ref{C24} and Theorem \ref{2.1} again, and substituting $p_1$ and $q_1$ for $p$ and $r$, yields the $18$th row of Table \ref{tab2}.\\
\textbf{II(B):} Suppose $v= 1$, then $x =  p$ and $y = 2$. We have $h_{K_2} = 1$, and by Corollary \ref{C24} we see that $2 \mid h_{K_1}$. Thus, applying Theorem \ref{2.1}, and substituting $p_1$ for $p$, we obtain the $19$th row of Table \ref{tab2}.\\
\textbf{Subcase (b):} Assume that $p \equiv 3 \pmod 4$. Applying Corollary \ref{C25}, we find that  $h_{K_3}$ is odd.\\
\textbf{I:} Assume that $x = 2pv $ and $y= v$.\\
\textbf{I(A):} Let $v = r$, where $r$ is an odd prime. Then, $x =  2pr$ and $y = r$. For the case $r \equiv 1 \pmod 4$, Corollary \ref{C24} and Theorem \ref{2.1} yield the $20$th row of Table \ref{tab2} by taking $p_1$ and $q_1$ for $r$ and $p$, respectively.\\

For the case $r \equiv 3 \pmod 4$, Corollary \ref{C24} and Theorem \ref{2.1} give: $$[h_{K_1},\ h_{K_2},\ h_{K_3},\ q(K)] \in \{[4,\ 1,\ 1,\ 2], [8,\ 1,\ 1,\ 1]\}.$$  This establishes the $21$st row of Table \ref{tab2} by substituting $q_1$ and $q_2$ for $r$ and $p$, respectively.\\
\textbf{I(B):} Suppose $v= 1$, then $x =  2p$ and $y = 1$. From this case, we obtain the $22$nd row of Table \ref{tab2} by substituting $q_1$ for $p$.\\

\noindent
\textbf{II:}  Assume that $x = pv$ and $y = 2v$.\\
\textbf{II(A):} Let $v = r$, where $r$ is an odd prime. Then, $x =  pr$ and $y = 2r$. The odd prime $r$ can satisfy either $r \equiv 1 \pmod 4$ or $r \equiv 3 \pmod 4$. First, consider  $r \equiv 1 \pmod 4$. Applying Corollary \ref{C24} and Theorem \ref{2.1} yields the $23$rd row of Table \ref{tab2} by taking $p_1$ and $q_1$ for $r$ and $p$, respectively.

Now, consider the case $r \equiv 3 \pmod{4}$. By Corollary~\ref{C24}, we have $4 \mid h_{K_1}$ and $2 \mid h_{K_2}$. Applying Theorem~\ref{2.1} and relabeling $q_1$ and $q_2$ as $r$ and $p$, respectively, yields the last row of Table~\ref{tab2}.\\
\textbf{II(B):} Suppose $v = 1$, so that $x = p$ and $y = 2$. In this situation, $h_{K_2} = 1$. By Corollary~\ref{C24}, $h_{K_1}$ must be odd. However, applying Theorem~\ref{2.1} to obtain $h_K = 4$ leads to $q(K) > 2$, which is not possible. Consequently, this case is eliminated.

Conversely, using Theorem \ref{2.1}, it is obvious that each field $K$ listed in Table \ref{tab2} has class number $4$. This completes the proof.
\end{proof}

\begin{prop} \label{P36} Let $K=  \Q (\sqrt{-x}, \sqrt{-y})$ such that $d_{xy}$ has exactly three distinct prime factors. Then, $h_K = 4$ if and only if $K$ is one of the fields that appeared in the following table:
\begin{center}
  $$
\begin{array}{|c|l|c|}
\hline
\# & \text{Field } K & [h_{K_1},\ h_{K_2},\ h_{K_3},\ q(K)] \\
\hline
1 & \mathbb{Q}(\sqrt{-p_1q_1}, \sqrt{-q_2}) & [2,\ 1,\ 2,\ 2],\ [4,\ 1,\ 2,\ 1],\ [2,\ 1,\ 4,\ 1]\\
\hline
2 & \mathbb{Q}(\sqrt{-2p_1q_1}, \sqrt{-1}) & [4,\ 1,\ 2,\ 1]\\
\hline
3 & \mathbb{Q}(\sqrt{-p_1q_1}, \sqrt{-2}) & [2,\ 1,\ 2,\ 2],\ [4,\ 1,\ 2,\ 1],\ [2,\ 1,\ 4,\ 1]\\
\hline
4 & \mathbb{Q}(\sqrt{-2p_1}, \sqrt{-q_1}) & [2,\ 1,\ 2,\ 2],\ [4,\ 1,\ 2,\ 1],\ [2,\ 1,\ 4,\ 1]\\
\hline

5 & \mathbb{Q}(\sqrt{-2q_1}, \sqrt{-p_1}) & [2,\ 2,\ 2,\ 1]\\
\hline
6 & \mathbb{Q}(\sqrt{-2q_1q_2}, \sqrt{-1}) & [4,\ 1,\ 2,\ 1]\\
\hline
7 & \mathbb{Q}(\sqrt{-q_1q_2}, \sqrt{-2}) & [4,\ 1,\ 2,\ 1]\\
\hline
8 & \mathbb{Q}(\sqrt{-2q_1}, \sqrt{-q_2}) & [2,\ 1,\ 2,\ 2],\ [4,\ 1,\ 2,\ 1],\ [2,\ 1,\ 4,\ 1]\\
\hline

\end{array}
$$
\captionof{table}{}
\label{tab3}
\end{center}
\end{prop}

 \begin{proof}
 Let $h_K =4$ and $d_{xy} = pqr $ or  $d_{xy} = 2pq $, where $p, q$ and $r$ are distinct odd primes. Applying Corollary \ref{C25} for the field $K_3 = \Q(\sqrt{d_{xy}})$, we find that  $2$ always divides $h_{K_3}$. The possibilities of $x$ and $y$, based on $d_{xy}$ and fields $K_1$  and $K_2$, are as follows: 
 \begin{itemize}
    \item  $x = pqrv $ and $y= v$,
    \item $x = pqv$ and $y = rv$,
     \item  $x = 2pqv $ and $y= v$,
      \item $x = pqv$ and $y = 2v$,
    \item $x = 2pv$ and $y = qv$,
   \end{itemize}
where $v$ is some square-free positive integer. Since $h_K = 4$ and $2 \mid h_{K_3}$,  Corollary \ref{C24} and Theorem \ref{2.1}  ensure that $v = 1$. Thus, we have: 
\begin{itemize}
    \item  $x = pqr $ and $y= 1$,
    \item $x = pq$ and $y = r$,
     \item  $x = 2pq $ and $y= 1$,
      \item $x = pq$ and $y = 2$,
    \item $x = 2p$ and $y = q$.
   \end{itemize}
We have $d_{xy} \equiv 1, 2, 3 \pmod 4$. We now consider each case of $d_{xy}$ along with each possible pair $(x, y)$. Note that for $d_{xy} \equiv 1, 3 \pmod 4$, the first two pairs $(x, y)$ are possible, whereas for  $d_{xy} \equiv 2 \pmod 4$ the last three pairs $(x, y)$ are possible.\\

\noindent
\textbf{Case (i):} Let $d_{xy} \equiv 1 \pmod 4  $. Based on this and the possible values of $(x, y)$ in our case, the following three subcases are sufficient:  $p, q, r \equiv 1 \pmod 4$; or $p, q  \equiv 3 \pmod 4$ and $r \equiv 1 \pmod 4$; or $p, r  \equiv 3 \pmod 4$ and $q \equiv 1 \pmod 4$.

\noindent
\textbf{Subcase (a):} Consider the situation where $p, q, r \equiv 1 \pmod 4$. In this case, from Corollary \ref{C25}, we have $4 \mid h_{K_3}$.\\
\noindent
\textbf{I:} Let $x = pqr $ and $y= 1$. In this case, by Corollary \ref{C24}, $4 \mid h_{K_1}$. Consequently, Theorem \ref{2.1} implies $h_K \neq 4$. This case is not possible.\\
\noindent
\textbf{II:}  Let $x = pq$ and $y = r$. In this case, by Corollary \ref{C24}, $4 \mid h_{K_1}$ and $2 \mid h_{K_2}$. Applying Theorem \ref{2.1}, we get $h_K \neq 4$. Thus, this case is also not possible.\\

\noindent
\textbf{Subcase (b):} Let $p, q  \equiv 3 \pmod 4$ and $r \equiv 1 \pmod 4$. We have $2 \mid h_{K_3}$.\\
\textbf{I:} Let $x = pqr $ and $y= 1$. 
Since $x$ has three distinct prime factors and $x \equiv 1 \pmod 4$, Corollary \ref{C24} implies that $8 \mid h_{K_1}$. Thus, Theorem \ref{2.1} guarantees that $h_K \neq 4$. Consequently, we discard this case. \\

\noindent
\textbf{II:}  Let $x = pq$ and $y = r$.
Here, $x$ has two distinct prime factors and $x \equiv 1 \pmod 4$. Using Corollary \ref{C24}, we have $4 \mid h_{K_1}$. Similar arguments show that $2 \mid h_{K_2}$.  Using these facts in Theorem \ref{2.1}, we have $h_K \neq 4$.\\

\noindent
\textbf{Subcase (c):} Let $p, r  \equiv 3 \pmod 4$ and $q \equiv 1 \pmod 4$. We have $2 \mid h_{K_3}$.\\
\textbf{I:} Let $x = pqr $ and $y = 1$. One can readily see that this case is not possible. \\

\noindent
\textbf{II:}  Let $x = pq$ and $y = r$. Corollary \ref{C24} implies that $2 \mid h_{K_1}$ and $h_{K_2}$ is odd. Thus, Theorem \ref{2.1} gives: $$[h_{K_1},\ h_{K_2},\ h_{K_3},\ q(K)] \in \{[2, 1, 2, 2], [4, 1, 2, 1],  [2, 1, 4, 1]\}. $$ Replacing $q, p, r$ with $p_1, q_1, q_2$, respectively, we obtain the first row of Table \ref{tab3}.\\

\noindent
\textbf{Case (ii):} Consider the case where $d_{xy} = pqr \equiv 3 \pmod 4  $. Since $\Q(\sqrt{-x})$ and $\Q(\sqrt{-y})$ are imaginary number fields, it suffices to consider the following three possible subcases for $(x, y)$:  $p, q, r \equiv 3 \pmod 4$; or $p, q  \equiv 1 \pmod 4$ and $r \equiv 3 \pmod 4$; or $p, r  \equiv 1 \pmod 4$ and $q \equiv 3 \pmod 4$. Also, note that Corollary \ref{C25} implies that $4 \mid h_{K_3}$.

\noindent
\textbf{Subcase (a):} Let $p, q, r \equiv 3 \pmod 4$.\\
\noindent
\textbf{I:} Let $x = pqr $ and $y= 1$.  By Corollary \ref{C24}, $4 \mid h_{K_1}$. Consequently, Theorem \ref{2.1} implies that $h_K \neq 4$. This case is not possible.\\
\noindent
\textbf{II:}  Let $x = pq$ and $y = r$.
Since $x$ has two distinct prime factors and $x \equiv 1 \pmod 4$, from Corollary \ref{C24} we have $4 \mid h_{K_1}$. Applying Theorem \ref{2.1}, we get $h_K \neq 4$. Thus, this case is also not possible.\\
\textbf{Subcase (b):} Let $p, q  \equiv 1 \pmod 4$ and $r \equiv 3 \pmod 4$. \\
\textbf{I:} Let $x = pqr $ and $y= 1$. It is easy to see that this case is not possible.\\
\textbf{II:}  Let $x = pq$ and $y = r$.\\
Here, $x$ has two distinct prime factors and $x \equiv 1 \pmod 4$. Using Corollary \ref{C24}, we have $4 \mid h_{K_1}$. Thus, applying Theorem \ref{2.1}, we conclude that this case is not possible.

\noindent
\textbf{Subcase (c):} Let $p, r \equiv 1 \pmod 4$ and $q \equiv 3 \pmod 4$. \\
\textbf{I:} Let $x = pqr $ and $y = 1$. By Corollary \ref{C24}, $4 \mid h_{K_1}$. Using this and the fact that $4 \mid h_{K_3}$ in Theorem \ref{2.1}, we get $h_{K} \neq 4$. Thus, we exclude this case.\\
\noindent
\textbf{II:}  Let $x = pq$ and $y = r$. By Corollary \ref{C24}, we have  $2 \mid h_{K_1}$  and $2 \mid h_{K_2}$. Also, we have already seen that $4 \mid h_{K_3}$. Thus, Theorem \ref{2.1} ensures that $h_K \neq 4$. Therefore, we eliminate this case.\\

\noindent
\textbf{Case (iii):} Consider $d_{xy} = 2pq$. This leads to the following subcases: either $p, q \equiv 1 \pmod 4$; or $p \equiv 1 \pmod 4$ and  $q \equiv 3 \pmod 4$; or $p \equiv 3 \pmod 4$ and  $q \equiv 1 \pmod 4$; or $p, q \equiv 3 \pmod 4$. Note that, by Corollary \ref{C25}, $2$ always divides $h_{K_3}$. \\
\textbf{Subcase (a):} Assume that $p, q \equiv 1 \pmod 4$. In this situation, Corollary \ref{C25} implies that $4 \mid h_{K_3}$. Thus, Theorem \ref{2.1} ensures that either $4$ divides $h_{K_1}$ or $h_{K_2}$, and $2$ divides both $h_{K_1}$ and $ h_{K_2}$, which leads to $h_K \neq 4$.\\
\textbf{I:} Consider $x = 2pq $ and $y= 1$.
Corollary \ref{C24} implies that $4 \mid h_{K_1}$. Thus, this case is not possible.\\
\textbf{II:}  Assume that $x = pq$  and  $y = 2$.  By Corollary \ref{C24}, we have $4 \mid h_{K_1}$. Thus, $h_K \neq 4$. Therefore, this case is ruled out.\\
\textbf{III:} Consider $x = 2p$ and $y = q$. In this case, Corollary \ref{C24} implies that $2 \mid h_{K_1}$ and $2 \mid h_{K_2}$. Thus, this case is also not possible.\\

\noindent
\textbf{Subcase (b):} Assume that $p\equiv 1 \pmod 4$ and  $q\equiv 3\pmod 4$.  By Corollary \ref{C25}, we have $2 \mid h_{K_3}$.\\ 
\textbf{I:} Consider $x = 2pq $ and $y= 1$. Then, Corollary \ref{C24} implies that $4 \mid h_{K_1}$, and it is known that the class number of $\Q(\sqrt{-1})$ is $1$. Thus, applying Theorem \ref{2.1}, we have $[h_{K_1},\ h_{K_2},\ h_{K_3},\ q(K)] = [4, 1, 2, 1]$. Using $p_1$ and $q_1$ for $p$ and $q$, respectively, we obtain the $2$nd row of Table \ref{tab3}.\\
\textbf{II:}Assume that $x = pq$  and  $y = 2$.
By Corollary \ref{C24}, we have $2\mid h_{K_1}$, and it is easy to see that $h_{K_2} = 1$. Thus, Theorem \ref{2.1} implies that $$[h_{K_1},\ h_{K_2},\ h_{K_3},\ q(K)] \in \{[2,\ 1,\ 2,\ 2],\ [4,\ 1,\ 2,\ 1],\ [2,\ 1,\ 4,\ 1]\}.$$ Replacing $p$ and $q$ with $p_1$ and $q_1$, respectively, we obtain the $3$rd row of Table \ref{tab3}.\\
\textbf{III:} Consider $x = 2p$ and $y = q$. Since $2 \mid h_{K_3}$,
  applying Corollary \ref{C24} and Theorem \ref{2.1} we find that $[h_{K_1},\ h_{K_2},\ h_{K_3},\ q(K)] \in \{[2, 1, 2, 2], [4, 1, 2, 1], [2, 1, 4, 1]\}$. Thus,  by taking $p_1$ and $q_1$ for $p$ and $q$, respectively, we obtain the $4$th row of Table \ref{tab3}.\\

\noindent
\textbf{Subcase (c):} Assume that $p\equiv 3 \pmod 4$ and  $q\equiv 1\pmod 4$.  By Corollary \ref{C25}, we have $2 \mid h_{K_3}$. \\
\textbf{I:} Consider $x = 2pq $ and $y= 1$. Corollary \ref{C24} implies that $4 \mid h_{K_1}$, and we know that $h_{K_2} = 1$. Thus, applying Theorem \ref{2.1}, we obtain that $[h_{K_1},\ h_{K_2},\ h_{K_3},\ q(K)] = [4, 1, 2, 1]$. Therefore, using $p_1$ and $q_1$ for $q$ and $p$, respectively, we again have the $2$nd row of Table \ref{tab3}.\\
\textbf{II:}  Assume that $x = pq$  and  $y = 2$.  We know that $h_{K_2} = 1$, and  Corollary \ref{C24} implies that $2\mid h_{K_1}$. Thus, Theorem \ref{2.1} implies that $$[h_{K_1},\ h_{K_2},\ h_{K_3},\ q(K)] \in \{[2,\ 1,\ 2,\ 2],\ [4,\ 1,\ 2,\ 1],\ [2,\ 1,\ 4,\ 1]\}.$$ Replacing $p$ and $q$ with $q_1$ and $p_1$, respectively, we again obtain the $4$th row of Table \ref{tab3}.\\
\textbf{III:} Consider $x = 2p$ and $y = q$.  By Corollary \ref{C24}, we have $2 \mid h_{K_1}$ and $2 \mid h_{K_2}$. Applying Theorem \ref{2.1}, we have $[h_{K_1},\ h_{K_2},\ h_{K_3},\ q(K)] = [2, 2, 2, 1] $. Thus, by taking $p_1$ and $q_1$ for $q$ and $p$, respectively, we obtain the $5$th row of Table \ref{tab3}.\\
\noindent
\textbf{Subcase (d):} Assume that $p\equiv 3 \pmod 4$ and  $q\equiv 3\pmod 4$.  By Corollary \ref{C25}, we have $2 \mid h_{K_3}$.\\ 
\textbf{I:} Consider $x = 2pq $ and $y= 1$. Corollary \ref{C24} implies that $4 \mid h_{K_1}$. We know that $h_{K_2} = 1$. Thus, applying Theorem \ref{2.1}, we obtain that $[h_{K_1},\ h_{K_2},\ h_{K_3},\ q(K)] = [4, 1, 2, 1]$. Therefore, using $q_1$ and $q_2$ for $p$ and $q$, respectively, we obtain the $6$th row of Table \ref{tab3}.\\
\textbf{II:}  Consider $x = pq$  and  $y = 2$. By Corollary \ref{C24}, we have $4\mid h_{K_1}$, and the class number of $\Q(\sqrt{-2})$ is $1$.  Thus, Theorem \ref{2.1} implies that the only possibility for the tuple $[h_{K_1},\ h_{K_2},\ h_{K_3},\ q(K)]$ is $[4,\ 1,\ 2,\ 1]$. Relabeling $p$ and $q$ with $q_1$ and $q_2$, respectively, we obtain the $7$th row of Table \ref{tab3}.\\
\textbf{III:} Let $x = 2p$ and $y = q$.
By Corollary \ref{C24}, we have $2 \mid h_{K_1}$  and $h_{K_2}$ is odd. Applying Theorem \ref{2.1}, we have $$[h_{K_1},\ h_{K_2},\ h_{K_3},\ q(K)] \in \{[2,\ 1,\ 2,\ 2],\ [4,\ 1,\ 2,\ 1],\ [2,\ 1,\ 4,\ 1]\}.$$ Thus, taking $q_1$ and $q_2$ for $p$ and $q$, respectively, we obtain the $8$th row of Table \ref{tab3}.\\

Conversely, let $K$ be as listed in Table \ref{tab3} with the mentioned values of $[h_{K_1}, h_{K_2}, h_{K_3}, q(K)]$. Then, by applying Theorem \ref{2.1}, we get $h_K = 4$. This completes the proof.
\end{proof} 
\begin{proof}[Proof of Theorem \ref{T4}] We explain here our method for listing such fields $K$ using Table \ref{tab1} and \texttt{SageMath} without bothering for the value of $q(K)$.  Consider the first row of Table \ref{tab1}. We have $K = \Q(\sqrt{-p_1p_2q_1}, \sqrt{-p_2q_1})$ and $[h_{K_1}, h_{K_2}, h_{K_3}, q(K)]= [4, 2, 1, 1]$.
Now, consider  two imaginary quadratic subfields $K_1 = \Q(\sqrt{-p_1p_2q_1})$ and $K_2 = \Q(\sqrt{-p_2q_1})$ of $K$. Choose $x =p_1p_2q_1$ and $y =p_2q_1$ from the table presented in Section $10$ such that $h_{K_1} =4$ and $h_{K_2} =2$. Then, list all the fields $K =\Q(\sqrt{-x}, \sqrt{-y})$ for which $h_K =4$ and $h_{K_3} = 1$ using \texttt{SageMath} by providing the input pair $(x, y)$. Note that if both $(x_1, y_1)$ and $(y_1, x_1)$ appear as possible values  of $(x, y)$ for a fixed tuple $[h_{K_1}, h_{K_2}, h_{K_3}, q(K)]$, then we must have $h_{K_1} = h_{K_2}$. In this case, $\Q(\sqrt{-x_1}, \sqrt{-y_1}) = \Q(\sqrt{-y_1}, \sqrt{-x_1})$, so only one of the pairs, either $(x_1, y_1)$ or $(y_1, x_1)$, should be counted.\\

Note that during this process, we have not used the value of $q(K)$ at all. We only fixed the tuple $[h_{K_1}, h_{K_2}, h_{K_3}, q(K)]$, as it can have more than one value, such as in the 5th row of Table \ref{tab1}. Repeat this process for each case in each row. Apply the same procedure to the next table as well. This method is used in all cases to obtain a complete list of such fields $K$.\\

Using Table \ref{tab1} and \texttt{SageMath} as explained above, we obtain $67$ fields $K$ with $h_K =4$, of which $52$ fields have cyclic class groups. These fields are presented in Table \ref{tab8} of Section $7$. Applying the same procedure to Table \ref{tab2} yields $180$  such fields with $128$ having cyclic class groups, as shown in Table \ref{tab9}.  Similarly, using \texttt{SageMath} for Table \ref{tab3}, we obtain exactly $161$  such fields, of which $116$ fields have cyclic class groups. These fields are listed in Table \ref{tab10} of Section $7$.  Thus, in total, there are exactly 408 fields $K$ with $ h_K = 4$, of which $296$ have cyclic class groups. This completes the proof.
\end{proof}
 \section{Proof of Theorem \ref{T6}} The first result we prove in this section is the following result.
\begin{prop} \label{P41} Let $K = \Q(\sqrt{-x}, \sqrt{-y})$ be an imaginary bicyclic biquadratic number field. If $d_{xy}$ has at least four distinct prime factors, then $h_K \neq 6$.
\end{prop}
\begin{proof}
Here, $K_1 = \Q(\sqrt{-x})$, $K_2 = \Q(\sqrt{-y})$ and $K_3 = \Q(\sqrt{d_{xy}})$. Since $d_{xy}$ has at least four distinct prime factors, Corollary \ref{C25} implies that $4$ divides $h_{K_3}$. Consequently, by Proposition \ref{P21}, we have $4 \mid h_{K}$, and hence $h_K \neq 6$. This completes the proof.

\end{proof}
Proposition \ref{P41} guarantees that if $h_K =6$, then $d_{xy}$ has at most three prime factors. Using arguments similar to those used for $h_K = 4$ in the previous section, one can prove the following result. However, for completeness, we provide the full proof.
\begin{prop} \label{P42} Let \( K = \mathbb{Q}(\sqrt{-x}, \sqrt{-y}) \) be an imaginary bicyclic biquadratic number field. 
\begin{enumerate}
    \item Assume that $d_{xy}$ is a prime. Then, $h_K = 6$ if and only if $K$ appears in the table below: \vspace{0.5cm}
 \renewcommand{\arraystretch}{1.5}
\begin{center}
 \begin{tabular}{|c|c|m{20.1em}|}\hline
\# & \text{Field } $K$ & $[h_{K_1},\ h_{K_2},\ h_{K_3},\ q(K)]$ \\
\hline
$1$ & $\mathbb{Q}(\sqrt{-p_1q_1}, \sqrt{-q_1})$ & $[2,\ 3,\ 1,\ 2],\ [2,\ 1,\ 3,\ 2],\ [6,\ 1,\ 1,\ 2],$ $ [4,\ 3,\ 1,\ 1],\ [4,\ 1,\ 3,\ 1],\ [12,\ 1,\ 1,\ 1]$\\
\hline
$2$ & $\mathbb{Q}(\sqrt{-2p_1}, \sqrt{-2})$ & $[2,\ 1,\ 3,\ 2],\ [6,\ 1,\ 1,\ 2],\ [4,\ 1,\ 3,\ 1],$ $\ [12,\ 1,\ 1,\ 1]$\\
\hline
$3$ & $\mathbb{Q}(\sqrt{-p_1}, \sqrt{-1})$ & $[2,\ 1,\ 3,\ 2],\ [6,\ 1,\ 1,\ 2],\ [4,\ 1,\ 3,\ 1],$ $\ [12,\ 1,\ 1,\ 1]$\\
\hline

$4$ & $\mathbb{Q}(\sqrt{-p_1q_1}, \sqrt{-p_1})$ & $[6,\ 2,\ 1,\ 1],\ [2,\ 6,\ 1,\ 1],\ [2,\ 2,\ 3,\ 1]$\\
\hline
$5$ & $\mathbb{Q}(\sqrt{-q_1q_2}, \sqrt{-q_1})$& $[4,\ 3,\ 1,\ 1],\ [4,\ 1,\ 3,\ 1],\ [12,\ 1,\ 1,\ 1]$\\
\hline
$6$ & $\mathbb{Q}(\sqrt{-2q_1}, \sqrt{-2})$ & $[2,\ 1,\ 3,\ 2],\ [6,\ 1,\ 1,\ 2],\ [4,\ 1,\ 3,\ 1],$ $\ [12,\ 1,\ 1,\ 1]$\\
\hline
$7$ & $\mathbb{Q}(\sqrt{-2p_1}, \sqrt{-p_1})$ & $ [6,\ 2,\ 1,\ 1], \ [2,\ 6,\ 1,\ 1]$\\
\hline
$8$ & $\mathbb{Q}(\sqrt{-2q_1}, \sqrt{-q_1})$ & $[2,\ 3,\ 1,\ 2],\ [6,\ 1,\ 1,\ 2],\ [4,\ 3,\ 1,\ 1],$ $\ [12,\ 1,\ 1,\ 1]$\\
\hline
\end{tabular}
\captionof{table}{}
\label{tab4}
\end{center}
\vspace{0.6cm}
\item Assume that $d_{xy}$ has exactly two distinct prime factors. Then, $h_K = 6$ if and only if $K$ appears in the table below:
\begin{center}
 \begin{longtable*}{|c|c|m{19.1em}|}  \hline
\# & \text{Field } $K$ & $[h_{K_1},\ h_{K_2},\ h_{K_3},\ q(K)]$ \\
\hline 
$1$ & $\mathbb{Q}(\sqrt{-q_1q_2q_3}, \sqrt{-q_1})$ & $[4,\ 1,\ 3,\ 1],\ [4,\ 3,\ 1,\ 1],\ [12,\ 1,\ 1,\ 1]$\\
\hline
$2$ & $\mathbb{Q}(\sqrt{-2q_1q_2}, \sqrt{-2})$ & $[4,\ 1,\ 3,\ 1],\ [12,\ 1,\ 1,\ 1]$\\
\hline
$3$ & $\mathbb{Q}(\sqrt{-q_1q_2}, \sqrt{-1})$ & $[4,\ 1,\ 3,\ 1],\ [12,\ 1,\ 1,\ 1]$\\
\hline
$4$ & $\mathbb{Q}(\sqrt{-p_1q_1}, \sqrt{-p_1q_2})$ & $[2,\ 6,\ 1,\ 1],\ [2,\ 2,\ 3,\ 1]$\\
\hline
$5$ & $\mathbb{Q}(\sqrt{-2q_1}, \sqrt{-2q_2})$ & $[2,\ 2,\ 3,\ 1],\ [6,\ 2,\ 1,\ 1]$\\
\hline
$6$ & $\mathbb{Q}(\sqrt{-p_1q_1}, \sqrt{-1})$ & $[6,\ 1,\ 2,\ 1],\ [2,\ 1,\ 6,\ 1]$\\
\hline
$7$ & $\mathbb{Q}(\sqrt{-p_1}, \sqrt{-q_1})$ & $[6,\ 1,\ 2,\ 1],\ [2,\ 1,\ 6,\ 1],\ [2,\ 3,\ 2,\ 1]$\\
\hline
$8$ & $\mathbb{Q}(\sqrt{-2p_1}, \sqrt{-1})$ & $[6,\ 1,\ 2,\ 1],\ [2,\ 1,\ 6,\ 1]$\\
\hline
$9$ & $\mathbb{Q}(\sqrt{-p_1}, \sqrt{-2})$ & $[6,\ 1,\ 2,\ 1],\ [2,\ 1,\ 6,\ 1]$\\
\hline
$10$ & $\mathbb{Q}(\sqrt{-2q_1q_2}, \sqrt{-q_1})$ & $[4,\ 3,\ 1,\ 1],\ [4,\ 1,\ 3,\ 1],\ [12,\ 1,\ 1,\ 1]$\\
\hline
$11$ & $\mathbb{Q}(\sqrt{-2q_1}, \sqrt{-1})$ & $[2,\ 1,\ 3,\ 2],\ [6,\ 1,\ 1,\ 2],\ [4,\ 1,\ 3,\ 1],$ $\ [12,\ 1,\ 1,\ 1]$\\
\hline
$12$ & $\mathbb{Q}(\sqrt{-p_1q_1}, \sqrt{-2p_1})$ & $[6,\ 2,\ 1,\ 1],\ [2,\ 6,\ 1,\ 1],\ [2,\ 2,\ 3,\ 1]$\\
\hline
\end{longtable*}
\captionof{table}{}
\label{tab5}
\end{center}
\item Assume that $d_{xy}$ has exactly three distinct prime factors. Then, $h_K = 6$ if and only if $K$ is one of the fields listed in the table below:
\begin{center}
 $$
\begin{array}{|c|l|c|}
\hline
\# & \text{Field } K & [h_{K_1},\ h_{K_2},\ h_{K_3},\ q(K)] \\
\hline
1 & \mathbb{Q}(\sqrt{-p_1q_1}, \sqrt{-q_2}) & [2,\ 3,\ 2,\ 1],\ [6,\ 1,\ 2,\ 1],\ [2,\ 1,\ 6,\ 1]\\
\hline

2 & \mathbb{Q}(\sqrt{-p_1q_1}, \sqrt{-2}) & [6,\ 1,\ 2,\ 1],\ [2,\ 1,\ 6,\ 1]\\
\hline
3 & \mathbb{Q}(\sqrt{-2p_1}, \sqrt{-q_1}) & [6,\ 1,\ 2,\ 1],\ [2,\ 1,\ 6,\ 1],\ [2,\ 3,\ 2,\ 1]\\
\hline
4 & \mathbb{Q}(\sqrt{-2q_1}, \sqrt{-q_2}) & [2,\ 3,\ 2,\ 1],\ [6,\ 1,\ 2,\ 1],\ [2,\ 1,\ 6,\ 1]\\
\hline
\end{array}
$$
\captionof{table}{}
\label{tab6}
\end{center}
\end{enumerate}
\end{prop}
\begin{proof}[Proof of Proposition \ref{P42}] We first observe that if $K$ appears in Tables~\ref{tab4}, \ref{tab5}, or \ref{tab6} as presented in Proposition~\ref{P42}, then by Theorem~\ref{2.1}, it follows that $h_{K} = 6$. We now proceed to demonstrate, case by case, that these tables contain all possible tuples $[h_{K_1},\ h_{K_2},\ h_{K_3},\ q(K)]$ for which $h_K = 6$.\\

\noindent
\textbf{(1).} Suppose $h_K = 6$ and $d_{xy} = p$ or $d_{xy} = 2$, where $p$ is an odd prime. Then the possible choices for $(x, y)$ are given by either $x = p v$, $y = v$, or $x = 2v$, $y = v$, where $v$ is a square-free positive integer.  Corollary \ref{C25} confirms that $h_{K_3}$ is always odd. As $h_K = 6$, from Corollary \ref{C24} and Theorem \ref{2.1}, it follows that $v$ cannot have more than one prime factor. Thus, the value of $v$ can be an odd prime $q$ or equal to $2$ or $1$. However, for $d_{xy} = 2$, we have $v \neq 2$.\\
\noindent
\textbf{Case (i):} Let $d_{xy} = p \equiv 1 \pmod 4$, and
consider $x = pv$ and $y = v$.\\
\textbf{A:} Assume that $v = q \equiv 1 \pmod 4$. By Corollary \ref{C24}, we have $4 \mid h_{K_1}$ and $2 \mid h_{K_2}$. Thus, Theorem \ref{2.1} implies that $h_K \neq 6$.\\
\textbf{B:} Consider $v = q \equiv 3 \pmod 4$. Corollary \ref{C24} implies that $2 \mid h_{K_1}$ and  $h_{K_2}$ is odd. Therefore, Theorem \ref{2.1} gives: \[
[h_{K_1},\ h_{K_2},\ h_{K_3},\ q(K)] \in\left\{
\begin{aligned}
&[2, 3, 1, 2],\ [2, 1, 3, 2],\ [6, 1, 1, 2],\\
&[4, 3, 1, 1],\ [4, 1, 3, 1],\ [12, 1, 1, 1]
\end{aligned} \right\}.
\]
This establishes the first row of Table \ref{tab4} by substituting $p_1$ and $q_1$ for $p$ and $q$, respectively.\\
\textbf{C:} Let $v = 2$. Then, $x = 2p$ and $y = 2$. We have $h_{K_2} = 1$, and by Corollary \ref{C24}  we see that  $2 \mid h_{K_1}$. Thus, Theorem \ref{2.1} implies that $$[h_{K_1},\ h_{K_2},\ h_{K_3},\ q(K)] \in \{[2,\ 1,\ 3,\ 2],\ [6,\ 1,\ 1,\ 2],\ [4,\ 1,\ 3,\ 1], \ [12,\ 1,\ 1,\ 1]\}.$$ This gives the second row of Table \ref{tab4} by substituting $p_1$  for $p$.\\
\textbf{D:} Let $v = 1$. Then, $x = p$ and $y = 1$. Then, $h_{K_2} = 1$, and  Corollary \ref{C24} implies that $2 \mid h_{K_1}$. Thus, applying Theorem \ref{2.1} and substituting $p_1$  for $p$, we obtain the $3$rd row of Table \ref{tab4}.\\

\noindent
\noindent
\textbf{Case (ii):} Let $d_{xy} = p \equiv 3 \pmod{4}$, with $x = p v$ and $y = v$.\\
\textbf{A:} Let $v = q \equiv 1 \pmod 4$. We have $x = pq$ and $y = q$. By Corollary \ref{C24}, $2 \mid h_{K_1}$ and $2 \mid h_{K_2}$. Therefore, Theorem \ref{2.1} guarantees that  $$[h_{K_1},\ h_{K_2},\ h_{K_3},\ q(K)] \in \{[2, 2, 3, 1], [6, 2, 1, 1],  [2,  6,  1, 1]\}.$$
Thus, relabeling $p$ and $q$ with $q_1$ and $p_1$, respectively, yield the $4$th row of Table \ref{tab4}.\\
\textbf{B:} Consider $v = q \equiv 3 \pmod 4$. From Corollary \ref{C24}, it follows that $4\mid h_{K_1}$ and $h_{K_2}$ is odd. Thus applying Theorem \ref{2.1} and taking $q = q_1$, $p = q_2$, we obtain the $5$th row of Table \ref{tab4}.\\
\textbf{C:} Let $v = 2$. Then, $x = 2p$ and $y = 2$. We know that $h_{K_2} = 1$. By Corollary \ref{C24}, we have $2 \mid h_{K_1}$. Thus, from Theorem \ref{2.1} and by setting $p = q_1$,  we obtain the $6$th row of Table \ref{tab4}.\\
\textbf{D:} Let $v = 1$. Then, $x = p$ and $y = 1$. We have $h_{K_2} = 1$. Corollary \ref{C24} implies that $h_{K_1}$ is odd. Thus, Theorem \ref{2.1} ensures that $h_K = 6$ only if $q(K) >2$, which is not possible. Hence, we discard this situation.\\

\noindent
\textbf{Case (iii)}: Consider $d_{xy} = 2$, and let $x = 2v$ and $y = v$. We have $h_{K_3} = 1$.\\
\textbf{A:} Assume that $v = q \equiv 1 \pmod 4$. Then, $x = 2q$ and $y = q \equiv 1\pmod 4$. From Corollary \ref{C24}, we obtain that $2 \mid h_{K_1}$ and $2 \mid h_{K_2}$. Therefore, applying Theorem \ref{2.1} and substituting $p_1$ for $q$ yields the $7$th row of Table \ref{tab4}.\\
\textbf{B:} Suppose  $v = q \equiv 3 \pmod 4$. Then, $x = 2q$ and $y = q \equiv 3\pmod 4$. From Corollary \ref{C24}, we obtain that $2 \mid h_{K_1}$ and $ h_{K_2}$ is odd. Thus, applying Theorem \ref{2.1} and substituting $q_1$ for $q$ yields the $8$th row of Table \ref{tab4}.\\
\textbf{C:} Let $v = 1$. Then, $x = 2$ and $y =1$. We have $h_{K_i} = 1$ for $i = 1, 2, 3.$ Using these values in Theorem \ref{2.1} to obtain $h_K = 6$, we conclude that $q(K) > 2$. This is not possible, and thus we exclude this case. This completes the proof of Case (1).\\

\noindent
\textbf{(2).} Let $h_K = 6$ and  $d_{xy} = pq$ or $d_{xy} = 2p$, where $p$ and $q$ are odd primes. Since $K_1$ and $K_2$ are imaginary quadratic fields,  the possible choices of $x$ and $y$ for $d_{xy} = pq$ are as follows:
\begin{itemize}
    \item $x = pqv$ and $y = v$,
    \item $x = pv$ and $y = qv$,
\end{itemize}
and for $d_{xy} = 2p$ we have 
\begin{itemize}
    \item $x = 2pv$ and $y = v$,
    \item $x = pv$ and $y = 2v$,
\end{itemize}
where $v$ is a square-free positive integer. Since $h_K = 6$, it follows from Corollary~\ref{C24} and Theorem~\ref{2.1} that $v$ must be either an odd prime $r$, or equal to $2$ or $1$. However, in the case $d_{xy} = 2p$, we have $v \neq 2$.\\
\textbf{Case (i):} Let $d_{xy} = pq\equiv 1 \pmod 4$. This leads to two subcases: $p, q \equiv 1 \pmod 4$ or $p, q \equiv 3 \pmod 4$. Corollary \ref{C24}, Corollary \ref{C25} and Theorem \ref{2.1} imply that the case $p, q \equiv 1 \pmod 4$ gives $h_{K} \neq 6$. Thus, we exclude this case.\\

Consider the case where $p \equiv 3 \pmod 4$ and $q \equiv 3 \pmod 4$.  Then, by Corollary~\ref{C25}, it follows that $h_{K_3}$ is odd.\\
\textbf{I:} Assume that  $x = pqv$ and $y = v$.\\
\textbf{I(A):} Let $v = r \equiv 1 \pmod 4$ be a prime. Then, $x = pqr$ and $y = r$. Using Corollary \ref{C24} and Theorem \ref{2.1}, we exclude this case.\\
\textbf{I(B):} Let $v = r \equiv 3 \pmod 4$ be a prime. Then, $x = pqr$ and $y = r$. Set $q_1 = r$, $q_2 = q$ and $q_3 = p$. By Corollary \ref{C24}, we see that $4 \mid h_{K_1}$ and $h_{K_2}$ is odd. Thus, the first row of Table \ref{tab5} follows from Theorem \ref{2.1}. \\
\textbf{I(C):} Let $v = 2$. Then, $x = 2pq$ and $y = 2$. Set $q_1 = p$ and $q_2 = q$. We have $h_{K_2} = 1$, and Corollary \ref{C24} implies that $4 \mid h_{K_1}$. Thus, applying Theorem \ref{2.1}, we obtain the second row of Table \ref{tab5}.\\
\textbf{I(D):} Let $v = 1$. Then, $x = pq$ and $y = 1$. Set $q_1 = p$ and $q_2 = q$. This case gives the $3$rd row of Table \ref{tab5}.\\

\noindent
\textbf{II:} Let $x = pv$ and $y = qv$. \\
\textbf{II(A):} Let $v = r \equiv 1 \pmod 4$ be a prime. Then, $x = pr$ and $y = qr$. Set $p_1 = r$, $q_1 = p$ and $q_2 = q$. By Corollary \ref{C24}, we have  $2 \mid h_{K_1}$ and $2 \mid h_{K_2}$. Thus, Theorem \ref{2.1} gives:  $$[h_{K_1},\ h_{K_2},\ h_{K_3},\ q(K)] \in \{[2, 2, 3, 1], [6, 2, 1, 1], [2,  6,  1, 1]\}.$$ 

Consider the tuples $[6, 2, 1, 1]$ and $ [2,  6,  1, 1]$. By reversing the order of the chosen pairs $(x, y)$ of the form $(p_1q_1, p_1q_2)$, where $h_{K_1} = 6$ and $h_{K_2} = 2$, we obtain pairs of the form $(p_1q_2, p_1q_1)$, for which $h_{K_1} = 2$ and $h_{K_2} = 6$, and vice versa. Thus, the fields $K$ obtained in both cases correspond to the same list. Hence, we consider $[h_{K_1},\ h_{K_2},\ h_{K_3},\ q(K)] \in \{[2, 2, 3, 1], [2, 6, 1, 1]\}$, which yield the $4$-th row of Table~\ref{tab5}.
.\\
\textbf{II(B):} Let $v = r \equiv 3 \pmod 4$ be a prime. Then, $x = pr$ and $y = qr$.  By Corollary \ref{C24}, we see that $4 \mid h_{K_1}$ and $4 \mid h_{K_2}$. Therefore,  Theorem \ref{2.1} implies that $h_K \neq 6$. Thus, we discard this case.\\
\textbf{II(C):} Let $v = 2$. Then, $x = 2p$ and $y = 2q$. Set $q_1 = p$ and $q_2 = q$.
By Corollary \ref{C24}, $2 \mid h_{K_1}$ and $2 \mid h_{K_2}$. Now, the $5$th row of Table \ref{tab5} follows from Theorem \ref{2.1}. \\
\textbf{II(D):} Let $v = 1$. Then, $x = p$ and $y = q$. Set $q_1 = p$ and $q_2 = q$. Corollary \ref{C24} implies that $h_{K_1}$ and $h_{K_2}$ are odd, and we have seen that $h_{K_3}$ is also odd. Thus, for  $h_K = 6$, Theorem \ref{2.1} implies that $q(K)> 2$. This is not possible; therefore, we exclude this case.
\\
\textbf{Case (ii):} Let $d_{xy} = pq\equiv 3 \pmod 4$. This leads to two subcases: either $p\equiv 1 \pmod 4$ and  $q \equiv 3 \pmod 4$ or $p\equiv 3 \pmod 4$ and  $q \equiv 1 \pmod 4$. However, in our case, considering one subcase is sufficient. So, we assume the subcase $p\equiv 1 \pmod 4$ and  $q \equiv 3 \pmod 4$. 

Using Corollary \ref{C25}, we have $2 \mid h_{K_3}$. Now, Corollary \ref{C24} and Theorem \ref{2.1} ensure that, to obtain $h_{K} = 6$, we must have $v = 1$ in both  cases: either $(x,y) = (pqv,v)$ or $(x, y) = (pv, qv)$. Set $p_1 = p$ and $q_1 = q$. Thus, the case $x = p_1q_1$ and $y = 1$ gives the $6$th row of Table \ref{tab5}, whereas the case $x = p_1$ and $y = q_1$ yields the $7$th row of  Table \ref{tab5}.\\

\noindent
\textbf{Case (iii):} Let $d_{xy} = 2p$. The odd prime $p$ can satisfy $p \equiv 1 \pmod 4$ or $p \equiv 3 \pmod 4$. \\
\textbf{Subcase (a)} Let $p \equiv 1 \pmod 4$. Then, by Corollary \ref{C25}, we have $2 \mid h_{K_3}$. Applying Corollary \ref{C24} and Theorem \ref{2.1} to obtain $h_{K} = 6$, we conclude that $v = 1$ in both  cases: $x = 2pv$ and $y = v$, or $x = pv$ and $y = 2v$. Set $p_1 = p$. Then, the case $x = 2p_1$ and $y = 1$ corresponds to the $8$th row of Table \ref{tab5}, while the case $x = p_1$ and $y = 2$ gives the $9$th row of  Table \ref{tab5}.\\
\noindent
\textbf{Subcase (b)} Let $p \equiv 3 \pmod 4$. Then, by Corollary \ref{C25}, we find that $ h_{K_3}$ is odd.\\
\textbf{I:} Assume that $x = 2pv$ and $y = v$.\\
\textbf{I(A):} Let $v = r \equiv 1 \pmod{4}$ be a prime. Then, $x = 2 p r$ and $y = r$. Set $p_1 = r$ and $q_1 = p$. By Corollary~\ref{C24}, it follows that $4 \mid h_{K_1}$ and $2 \mid h_{K_2}$. Therefore, Theorem~\ref{2.1} implies that $h_K \neq 6$, and consequently, this case is not possible.\\
\textbf{I(B):} Let $v = r \equiv 3 \pmod 4$ be a prime. Then, $x = 2pr$ and $y = r$. Set  $q_1 = r$ and $q_2 = p$. Then, this case corresponds to the $10$th row of  Table \ref{tab5}.\\
\textbf{I(C):} Let $v = 1$. Then, $x = 2p$ and $y = 1$. Set  $q_1 = p$. From this case, we obtain the $11$th row of  Table \ref{tab5}.\\

\noindent
\textbf{II:} Consider $x = pv$ and $y = 2v$. It is easy to see that if $v = r$ is a prime, then $r \not \equiv 3 \pmod 4$.\\
\textbf{II(A):} Let $v = r \equiv 1 \pmod 4$ be a prime. Then, $x = pr$ and $y = 2r$. Set $p_1 = r$ and $q_1 = p$. By Corollary \ref{C24}, we have $2 \mid h_{K_1}$ and $2 \mid h_{K_2}$. Thus, applying Theorem \ref{2.1}, we obtain the $12$th row of  Table \ref{tab5}.\\
\textbf{II(B):} Let $v = 1$. Then, we have $x = p$ and $y = 2$. Set $q_1 = p$. Corollary \ref{C24} implies that $h_{K_1}$ and $h_{K_2}$ are odd. Thus, to obtain $h_K = 6$, Theorem \ref{2.1} implies that $q(K) > 2$. This is not possible; therefore, we exclude this case. This completes the proof of Case (2).\\

\noindent
\textbf{(3).}
Let $h_K =6$ and $d_{xy} = pqr$ or $d_{xy} = 2pq$, where $p$, $q$ and $r$ are odd primes. Since $K_1$ and $K_2$ are imaginary number fields, for $d_{xy} = pqr$ the possible values of $x$ and $y$ are as follows:
\begin{itemize}
    \item $x = pqrv$ and $y = v$,
    \item  $x = pqv$ and $y = r v$,
\end{itemize}
where $v$ is a square-free positive integer. Similarly, for $d_{xy} = 2pq$ we have 
\begin{itemize}
\item  $x = 2pqv $ and $y= v$,
     \item $x = pqv$ and $y = 2v$,
    \item $x = 2pv$ and $y = qv$,
\end{itemize}
where $v$ is a square-free positive integer. Corollary~\ref{C25} implies that $2$ always divides $h_{K_3}$. Since $h_K = 6$, Corollary~\ref{C24} and Theorem~\ref{2.1} together ensure that $v = 1$. 
Note that the value of $d_{xy}$ can satisfy either 
$
d_{xy} \equiv 1, 3 \pmod{4} \quad \text{or} \quad d_{xy} \equiv 2 \pmod{4}.$

Furthermore, if $d_{xy} \equiv  3 \pmod 4$, then by Corollary \ref{C25} we have $4 \mid h_{K_3}$. Consequently, Proposition \ref{P21} implies that $4\mid h_K$. This shows that $h_K \neq 6$. Therefore, this case is ruled out.

\noindent
\textbf{Case (i):} Assume that $d_{xy} = p q r \equiv 1 \pmod{4}$. Based on this and the possible values of $(x, y)$, it suffices to consider the following three subcases: (a) $p, q, r \equiv 1 \pmod{4}$; (b) $p, q \equiv 3 \pmod{4} \text{ and } r \equiv 1 \pmod{4}$; (c) $p, r \equiv 3 \pmod{4} \text{ and } q \equiv 1 \pmod{4}.$

\noindent
\textbf{Subcase (a):} Consider $p, q, r \equiv 1 \pmod 4$. Then, by Corollary \ref{C25}, $4 \mid h_{K_3}$. Applying Proposition \ref{P21}, we obtain $4 \mid h_K$. Consequently, $h_K \neq 6$, and thus this case is ruled out.\\
\noindent\textbf{Subcase (b):} Assume that $p, q \equiv 3 \pmod{4}$ and $r \equiv 1 \pmod{4}$. In this subcase, we have $2 \mid h_{K_3}$. 
Consider either $x = p q r$ and $y = 1$, or $x = p q$ and $y = r$. Then, by Corollary~\ref{C24} and Theorem~\ref{2.1}, it follows that $h_K \neq 6$. 
Therefore, this case is not possible.\\
\noindent
\textbf{Subcase (c):} Assume that $p, r \equiv 3 \pmod 4$ and $q \equiv 1 \pmod 4$. Then, by Corollary \ref{C25}, we have $2 \mid h_{K_3}$. Thus, applying  Corollary \ref{C24} and Theorem \ref{2.1}, we obtain that the case $x = pqr$ and $y = 1$ is not possible.\\

Assume that $x = pq$ and $y = r$. Then by Corollary \ref{C24}, we have $2 \mid h_{K_1}$ and $h_{K_2}$ is odd. Thus,  Theorem \ref{2.1} implies that 
$$[h_{K_1},\ h_{K_2},\ h_{K_3},\ q(K)] \in \{[2, 3, 2, 1], [6, 1, 2, 1], [2,  1,  6, 1]\}.$$ This establishes the first row of Table \ref{tab6} by substituting $p_1, q_1, q_2$ for $q, p, r$, respectively.\\
\noindent
\textbf{Case (ii):} Let $d_{xy} = 2pq$.\\
Since $2 \mid h_{K_3}$, applying Corollary \ref{C24} and Theorem \ref{2.1} we obtain that $h_K \neq 6$ for the case $x = 2pq $ and $y= 1$. Thus, we discard this case.\\

Furthermore, by Corollary \ref{C24} and Theorem \ref{2.1}, it follows that $h_K \neq6$ for the case $p \equiv 1 \pmod 4$ and $q \equiv 1 \pmod 4$. Thus, we will not consider this case further.\\
\textbf{Subcase (a):} Consider $p \equiv 1 \pmod 4$ and $q \equiv 3 \pmod 4$. Then, we have $2 \mid h_{K_3}$.\\
\textbf{I:} Let $x = pq$ and $y = 2$. Then, $h_{K_2} = 1$, and by Corollary \ref{C24} we have $2 \mid h_{K_1}$. Now, from Theorem \ref{2.1}, we find that 

$$[h_{K_1},\ h_{K_2},\ h_{K_3},\ q(K)] \in \{[6, 1, 2, 1], [2, 1, 6, 1]\}.$$ Thus, relabeling $p$ and $q$ as $p_1$ and $q_1$, respectively, yields the second row of Table~\ref{tab6}.\\
\textbf{II:} Assume that $x = 2p$ and $y = q$. From this case, we obtain the $3$rd row of Table \ref{tab6} by applying Corollary \ref{C24} and Theorem \ref{2.1}, and by substituting $p_1$ and $q_1$ for $p$ and $q$ respectively.

\noindent
\textbf{Subcase (b):} Consider $p \equiv 3 \pmod 4$ and $q \equiv 1 \pmod 4$. Then, we have $2 \mid h_{K_3}$.
If we consider $x = pq$ and $y = 2$, then we again obtain the second row of Table \ref{tab6}. However, by taking $x = 2p$ and $y = q$, we do not obtain $h_K = 6$. \\
\textbf{Subcase (c):} Consider $p \equiv 3 \pmod 4$ and $q \equiv 3 \pmod 4$.\\
\textbf{I:} Let $x = pq$ and $y = 2$. Then, $h_{K_2} = 1$, and by Corollary \ref{C24} we have $4 \mid h_{K_1}$. Thus, from Theorem \ref{2.1}, we conclude that $h_K \neq 6.$\\
\textbf{II:} Assume that $x = 2p$ and $y = q$. Then, by Corollary \ref{C24}, we have $2 \mid h_{K_1}$ and $h_{K_2}$ is odd. Thus, applying Theorem \ref{2.1} and substituting $q_1$ and $q_2$ for $p$ and $q$, we obtain the $4$th row of Table \ref{tab6}. This completes the proof of Case (3), and thus the proposition follows. 
\end{proof}
\begin{proof}[Proof of Theorem \ref{T6}] Note that, in Proposition~\ref{P41}, we proved that if $d_{xy}$ has at least four distinct prime factors, then no field $K$ satisfies $h_K = 6$. Therefore, the number fields $K$ with class number $6$ obtained from Proposition~\ref{P42} are the only imaginary bicyclic biquadratic number fields with class number $6$.

By applying a similar procedure, as outlined in the first and second paragraphs of the proof of Theorem~\ref{T4}, to the tables presented in Proposition~\ref{P42}, we obtain $552$ fields $K$ with class number $6$. These fields are listed in Section~$8$. This completes the proof.
\end{proof}
\begin{rmk}
    We note that using Proposition \ref{P42} and \texttt{SageMath}, one can list all fields $K$ with class number $2p$, where $p$ is an odd prime and $4p \leq 100$. This is achieved by replacing $3$, $6$ and $12$ with $p$, $2p$ and $4p$ respectively, in the possible values of the tuple $[h_{K_1},\ h_{K_2},\ h_{K_3},\ q(K)]$ that appeared in Proposition \ref{P42}.
\end{rmk}

\section{Proof of Theorem \ref{T7}}
Using the method described in this article, one can list all imaginary bicyclic biquadratic fields with class number $7$. However, to list such fields $K$, we use the following proposition recorded from \cite{PJ}.
\begin{prop}\label{7.1}
The imaginary bicyclic biquadratic fields with odd class numbers are as follows:
\begin{enumerate}
    \item $\mathbb{Q}(\sqrt{-q_1}, \sqrt{-p_1q_1})$ with $\left( \frac{p_1}{q_1} \right) = \left( \frac{q_1}{p_1} \right) = -1$,
    \item $\mathbb{Q}(\sqrt{-1}, \sqrt{-p_1}),$ $ \mathbb{Q}(\sqrt{-2}, \sqrt{-2p_1})$ with $p_1 \equiv 5 \pmod{8}$,
    \item $\mathbb{Q}(\sqrt{-q_1}, \sqrt{-2q_1})$ with $q_1 \equiv 3 \pmod{8}$,
    \item $\mathbb{Q}(\sqrt{-1}, \sqrt{-2})$,
    \item $\mathbb{Q}(\sqrt{-1}, \sqrt{-q_1})$,
    \item $\mathbb{Q}(\sqrt{-2}, \sqrt{-q_1})$,
    \item $\mathbb{Q}(\sqrt{-q_1}, \sqrt{-q_2})$,
\end{enumerate}
where $q_1, q_2 \equiv 3 \pmod 4$ and $p_1 \equiv 1 \pmod 4$ are distinct odd primes.
\end{prop}
Using Proposition \ref{7.1}, followed by Theorem~\ref{2.2} and then Theorem~\ref{2.1}, one can establish the following result.

\begin{prop}\label{P52} Let $K = \Q(\sqrt{-x}, \sqrt{-y})$. Then, $h_K =7$ if and only if $K$ is one of the fields listed in the table below:\\

\renewcommand{\arraystretch}{1.5} 
\begin{center}
$$
\begin{array}{|c|l|c|}
\hline
\# & \text{Field } K & [h_{K_1},\ h_{K_2},\ h_{K_3},\ q(K)] \\
\hline
1 & \mathbb{Q}(\sqrt{-q_1}, \sqrt{-p_1q_1}) & [1,\ 2,\ 7,\ 1],\ [7,\ 2,\ 1,\ 1],\ [1,\ 14,\ 1,\ 1]\\
\hline
2 & \mathbb{Q}(\sqrt{-1}, \sqrt{-p_1}) & [1,\ 2,\ 7,\ 1],\ [1,\ 14,\ 1,\ 1]\\
\hline
3 & \mathbb{Q}(\sqrt{-2}, \sqrt{-2p_1}) & [1,\ 2,\ 7,\ 1],\ [1,\ 14,\ 1,\ 1]\\
\hline
4 & \mathbb{Q}(\sqrt{-q_1}, \sqrt{-2q_1}) & [7,\ 2,\ 1,\ 1],\ [1,\ 14,\ 1,\ 1]\\
\hline
5 & \mathbb{Q}(\sqrt{-1}, \sqrt{-q_1}) & [1,\ 1,\ 7,\ 2],\ [1,\ 7,\ 1,\ 2]\\
\hline
6 & \mathbb{Q}(\sqrt{-2}, \sqrt{-q_1}) & [1,\ 1,\ 7,\ 2],\ [1,\ 7,\ 1,\ 2]\\
\hline
7 & \mathbb{Q}(\sqrt{-q_1}, \sqrt{-q_2}) & [1,\ 1,\ 7,\ 2],\ [1,\ 7,\ 1,\ 2]\\
\hline
\end{array}$$
\captionof{table}{}
\label{tab7}
\end{center}
\end{prop}
\begin{proof}[Proof of Theorem \ref{T7}] Using \texttt{SageMath} for each row of Table \ref{tab7} as described in the first and second paragraphs of the proof of Theorem~\ref{T4}, we obtain that there are exactly $277$ fields $K$ which have $h_K = 7$. These fields are presented in Section $9$. This completes the proof.
\end{proof}
\begin{rmk}
We wish to point out that using Proposition \ref{P52} and \texttt{SageMath}, one can list all fields $K$ with class number $p$, where $p$ is an odd prime and $2p \leq 100$. This can be achieved by replacing $7$ and $14$ with $p$ and $2p$, respectively, in the possible values of the tuple $[h_{K_1},\ h_{K_2},\ h_{K_3},\ q(K)]$  that appeared in Proposition \ref{P52}.
\end{rmk}

\section{Concluding Remarks}
We remark that if one considers a field $K$ of the form $\mathbb{Q}(\sqrt{-x}, \sqrt{y})$, then there exists a square-free positive integer $z$ such that $
\mathbb{Q}(\sqrt{-x}, \sqrt{y}) = \mathbb{Q}(\sqrt{-z}, \sqrt{y}).$
In this situation, the discriminants $d_{xy}$ and $d_{zy}$ may or may not be equal. If $d_{xy} \neq d_{zy}$, then by our method, the field $K$ can appear more than once. To avoid this redundancy, it is preferable to consider $K$ of the form $\mathbb{Q}(\sqrt{-x}, \sqrt{-y})$.

Several corollaries can be drawn from this article. For example, let $p$ be a prime. Then, the class number of $\mathbb{Q}(\sqrt{-p}, \sqrt{-2})$ cannot be $4$, and the class number of $\mathbb{Q}(\sqrt{-2p}, \sqrt{-1})$ equals $4$ for exactly two prime values of $p$. We also note that, analogous to Proposition~\ref{7.1}, one can classify all fields $K$ with even class number using the methodology used in this article.

\section{Fields $K$ with $h_K = 4$} 
{\bf The fields $K$ obtained from Table \ref{tab1} using \texttt{SageMath} are presented in the table below.}
\renewcommand{\arraystretch}{1.3}




127 & $\Q(\sqrt{-10},\sqrt{-19})$  & $\Z/{2\Z}$ & $1$ & $\Z/{2\Z}$ & $\Z/{2\Z} \times \Z/{2\Z}$ \\
\hline
128 & $\Q(\sqrt{-58},\sqrt{-7})$   & $\Z/{2\Z}$ & $1$ & $\Z/{2\Z}$ & $\Z/{2\Z} \times \Z/{2\Z}$ \\
\hline
129 & $\Q(\sqrt{-58},\sqrt{-67})$  & $\Z/{2\Z}$ & $1$ & $\Z/{2\Z}$ & $\Z/{2\Z} \times \Z/{2\Z}$ \\
\hline
 130 & $\Q(\sqrt{-34},\sqrt{-3})$   & $\Z/{4\Z}$ & $1$ & $\Z/{2\Z}$ & $\Z/{4\Z}$ \\
\hline
131 & $\Q(\sqrt{-34},\sqrt{-7})$   & $\Z/{4\Z}$ & $1$ & $\Z/{2\Z}$ & $\Z/{4\Z}$ \\
\hline
132 & $\Q(\sqrt{-34},\sqrt{-11})$  & $\Z/{4\Z}$ & $1$ & $\Z/{2\Z}$ & $\Z/{4\Z}$ \\
\hline
133 & $\Q(\sqrt{-82},\sqrt{-3})$   & $\Z/{4\Z}$ & $1$ & $\Z/{2\Z}$ & $\Z/{4\Z}$ \\
\hline
134 & $\Q(\sqrt{-82},\sqrt{-11})$  & $\Z/{4\Z}$ & $1$ & $\Z/{2\Z}$ & $\Z/{4\Z}$ \\
\hline
135 & $\Q(\sqrt{-82},\sqrt{-19})$  & $\Z/{4\Z}$ & $1$ & $\Z/{2\Z}$ & $\Z/{4\Z}$ \\
\hline
136 & $\Q(\sqrt{-82},\sqrt{-67})$  & $\Z/{4\Z}$ & $1$ & $\Z/{2\Z}$ & $\Z/{4\Z}$ \\
\hline
137 & $\Q(\sqrt{-6},\sqrt{-13})$   & $\Z/{2\Z}$ & $\Z/{2\Z}$ & $\Z/{2\Z}$ & $\Z/{2\Z} \times \Z/{2\Z}$ \\
\hline
138 & $\Q(\sqrt{-6},\sqrt{-37})$   & $\Z/{2\Z}$ & $\Z/{2\Z}$ & $\Z/{2\Z}$ & $\Z/{2\Z} \times \Z/{2\Z}$ \\
\hline
139 & $\Q(\sqrt{-22},\sqrt{-5})$   & $\Z/{2\Z}$ & $\Z/{2\Z}$ & $\Z/{2\Z}$ & $\Z/{2\Z} \times \Z/{2\Z}$ \\
\hline
140 & $\Q(\sqrt{-22},\sqrt{-37})$  & $\Z/{2\Z}$ & $\Z/{2\Z}$ & $\Z/{2\Z}$ & $\Z/{2\Z} \times \Z/{2\Z}$ \\
\hline
141 & $\Q(\sqrt{-42},\sqrt{-1})$ & $\Z/{2\Z} \times \Z/{2\Z}$ & $1$ & $\Z/{2\Z}$ & $\Z/{2\Z} \times \Z/{2\Z}$ \\
\hline
142 & $\Q(\sqrt{-21},\sqrt{-2})$ & $\Z/{2\Z} \times \Z/{2\Z}$ & $1$ & $\Z/{2\Z}$ & $\Z/{2\Z} \times \Z/{2\Z}$ \\
\hline
143 & $\Q(\sqrt{-93},\sqrt{-2})$ & $\Z/{2\Z} \times \Z/{2\Z}$ & $1$ & $\Z/{2\Z}$ & $\Z/{2\Z} \times \Z/{2\Z}$ \\
\hline
144 & $\Q(\sqrt{-133},\sqrt{-2})$ & $\Z/{2\Z} \times \Z/{2\Z}$ & $1$ & $\Z/{2\Z}$ & $\Z/{2\Z} \times \Z/{2\Z}$ \\
\hline
145 & $\Q(\sqrt{-6},\sqrt{-7})$ & $\Z/{2\Z}$ & $1$ & $\Z/{2\Z}$ & $\Z/{2\Z} \times \Z/{2\Z}$ \\
\hline
146 & $\Q(\sqrt{-22},\sqrt{-7})$ & $\Z/{2\Z}$ & $1$ & $\Z/{2\Z}$ & $\Z/{2\Z} \times \Z/{2\Z}$ \\
\hline
147 & $\Q(\sqrt{-14},\sqrt{-3})$   & $\Z/{4\Z}$ & $1$ & $\Z/{2\Z}$ & $\Z/{4\Z}$ \\
\hline
148 & $\Q(\sqrt{-14},\sqrt{-11})$  & $\Z/{4\Z}$ & $1$ & $\Z/{2\Z}$ & $\Z/{4\Z}$ \\
\hline
149 & $\Q(\sqrt{-14},\sqrt{-19})$  & $\Z/{4\Z}$ & $1$ & $\Z/{2\Z}$ & $\Z/{4\Z}$ \\
\hline
150 & $\Q(\sqrt{-14},\sqrt{-43})$  & $\Z/{4\Z}$ & $1$ & $\Z/{2\Z}$ & $\Z/{4\Z}$ \\
\hline
151 & $\Q(\sqrt{-14},\sqrt{-67})$  & $\Z/{4\Z}$ & $1$ & $\Z/{2\Z}$ & $\Z/{4\Z}$ \\
\hline
152 & $\Q(\sqrt{-14},\sqrt{-163})$ & $\Z/{4\Z}$ & $1$ & $\Z/{2\Z}$ & $\Z/{4\Z}$ \\
\hline
153 & $\Q(\sqrt{-46},\sqrt{-3})$   & $\Z/{4\Z}$ & $1$ & $\Z/{2\Z}$ & $\Z/{4\Z}$ \\
\hline
154 & $\Q(\sqrt{-46},\sqrt{-43})$  & $\Z/{4\Z}$ & $1$ & $\Z/{2\Z}$ & $\Z/{4\Z}$ \\
\hline
155 & $\Q(\sqrt{-46},\sqrt{-67})$  & $\Z/{4\Z}$ & $1$ & $\Z/{2\Z}$ & $\Z/{4\Z}$ \\
\hline
156 & $\Q(\sqrt{-46},\sqrt{-163})$ & $\Z/{4\Z}$ & $1$ & $\Z/{2\Z}$ & $\Z/{4\Z}$ \\
\hline
157 & $\Q(\sqrt{-142},\sqrt{-3})$  & $\Z/{4\Z}$ & $1$ & $\Z/{2\Z}$ & $\Z/{4\Z}$ \\
\hline
158 & $\Q(\sqrt{-142},\sqrt{-11})$ & $\Z/{4\Z}$ & $1$ & $\Z/{2\Z}$ & $\Z/{4\Z}$ \\
\hline
159 & $\Q(\sqrt{-142},\sqrt{-19})$ & $\Z/{4\Z}$ & $1$ & $\Z/{2\Z}$ & $\Z/{4\Z}$ \\
\hline
160 & $\Q(\sqrt{-142},\sqrt{-43})$ & $\Z/{4\Z}$ & $1$ & $\Z/{2\Z}$ & $\Z/{4\Z}$ \\
\hline
161 & $\Q(\sqrt{-142},\sqrt{-67})$ & $\Z/{4\Z}$ & $1$ & $\Z/{2\Z}$ & $\Z/{4\Z}$ \\
\hline
\end{longtable}
\end{center}
\vspace{0.3cm}
\section{Fields $K$ with $h_K = 6$}  \vspace{0.3cm}

Note that $h_K = 6$ implies $C\ell(K) \simeq \mathbb{Z}/6\mathbb{Z}$. Therefore, in this case, we omit the column for $C\ell(K)$ from all the tables presented in this section. The 1st, 2nd, 3rd,  4th and 5th columns of the tables are for the field numbering, the fields $K$, and the class groups $C\ell(K_1)$, $C\ell(K_2)$, and $C\ell(K_3)$, respectively. The last five columns are assigned similarly. The fields $K$ obtained from Table~\ref{tab4} using \texttt{SageMath} are presented in the table below. 
\vspace{0.4cm}
\begin{center}
\begin{longtable}[h!]{|c|c|c|c|c||c|c|c|c|c|} 
\caption{\textbf{ $d_{xy}$ is prime and $h_K = 6$}}\label{tab11}
\\
\hline
1  & $\Q(\sqrt{-155}, \sqrt{-31})$  & $\Z/{4\Z}$ & $\Z/{3\Z}$ & $1$ 
& 45 & $\Q(\sqrt{-515}, \sqrt{-5})$ & $\Z/{6\Z}$ & $\Z/{2\Z}$ & $1$ \\
\hline
2  & $\Q(\sqrt{-667}, \sqrt{-23})$  & $\Z/{4\Z}$ & $\Z/{3\Z}$ & $1$ 
&  46 & $\Q(\sqrt{-835}, \sqrt{-5})$ & $\Z/{6\Z}$ & $\Z/{2\Z}$ & $1$ \\
\hline
3  & $\Q(\sqrt{-1003}, \sqrt{-59})$ & $\Z/{4\Z}$ & $\Z/{3\Z}$ & $1$ 
&  47 & $\Q(\sqrt{-1147}, \sqrt{-37})$ & $\Z/{6\Z}$ & $\Z/{2\Z}$ & $1$ \\
\hline
4  & $\Q(\sqrt{-1411}, \sqrt{-83})$ & $\Z/{4\Z}$ & $\Z/{3\Z}$ & $1$ 
& 48 & $\Q(\sqrt{-1315}, \sqrt{-5})$ & $\Z/{6\Z}$ & $\Z/{2\Z}$ & $1$ \\
\hline
5 & $\Q(\sqrt{-327}, \sqrt{-3})$ & $\Z/{12\Z}$ & $1$ & $1$ 
&  49 & $\Q(\sqrt{-1915}, \sqrt{-5})$ & $\Z/{6\Z}$ & $\Z/{2\Z}$ & $1$\\
\hline
6 & $\Q(\sqrt{-543}, \sqrt{-3})$ & $\Z/{12\Z}$ & $1$ & $1$ 
& 50 & $\Q(\sqrt{-1963}, \sqrt{-13})$ & $\Z/{6\Z}$ & $\Z/{2\Z}$ & $1$ \\
\hline
7 & $\Q(\sqrt{-731}, \sqrt{-43})$ & $\Z/{12\Z}$ & $1$ & $1$ 
& 51 & $\Q(\sqrt{-2515}, \sqrt{-5})$ & $\Z/{6\Z}$ & $\Z/{2\Z}$ & $1$ \\
\hline
   
\end{longtable}
\end{center}
\newpage
{\bf The fields $K$ computed from Table~\ref{tab5} using \texttt{SageMath} are listed in the table below.}
\begin{table}[H]
\centering
\caption{\textbf{ $d_{xy}$ is the product of two distinct primes and $h_K = 6$}}\label{tab12}
\end{table}
\begin{minipage}{.46\linewidth}
\begin{center}
\begin{longtable*}[h!]{|c|c|c|c|c|}
\hline
1 & $\Q(\sqrt{-483}, \sqrt{-23})$ & $\Z/{2\Z}\times\Z/{2\Z} $ & $\Z/{3\Z}$ & $1$ \\
\hline
2 & $\Q(\sqrt{-231}, \sqrt{-3})$ & $\Z/{6\Z} \times \Z/{2\Z}$ & $1$ & $1$ \\ \hline
3 & $\Q(\sqrt{-231}, \sqrt{-7})$ & $\Z/{6\Z} \times \Z/{2\Z}$ & $1$ & $1$ \\ \hline
4 & $\Q(\sqrt{-231}, \sqrt{-11})$ & $\Z/{6\Z} \times \Z/{2\Z}$ & $1$ & $1$ \\ \hline
5 & $\Q(\sqrt{-1419}, \sqrt{-11})$ & $\Z/{6\Z} \times \Z/{2\Z}$ & $1$ & $1$ \\ \hline
6 & $\Q(\sqrt{-1419}, \sqrt{-43})$ & $\Z/{6\Z} \times \Z/{2\Z}$ & $1$ & $1$ \\ \hline
7 & $\Q(\sqrt{-1491}, \sqrt{-3})$ & $\Z/{6\Z} \times \Z/{2\Z}$ & $1$ & $1$ \\ \hline
8 & $\Q(\sqrt{-1491}, \sqrt{-7})$ & $\Z/{6\Z} \times \Z/{2\Z}$ & $1$ & $1$ \\
\hline

\end{longtable*}
\end{center}
\end{minipage}
\hfill
\begin{minipage}{0.46\linewidth}
\begin{center}
\begin{longtable*}[h!]{|c|c|c|c|c|}

\hline

36 & $\Q(\sqrt{-453}, \sqrt{-1})$   & $\Z/{6\Z} \times \Z/{2\Z}$ & $1$ & $1$ \\ \hline
37 & $\Q(\sqrt{-517}, \sqrt{-1})$   & $\Z/{6\Z} \times \Z/{2\Z}$ & $1$ & $1$ \\ \hline
38 & $\Q(\sqrt{-537}, \sqrt{-1})$   & $\Z/{6\Z} \times \Z/{2\Z}$ & $1$ & $1$ \\ \hline
39 & $\Q(\sqrt{-597}, \sqrt{-1})$   & $\Z/{6\Z} \times \Z/{2\Z}$ & $1$ & $1$ \\ \hline
40 & $\Q(\sqrt{-669}, \sqrt{-1})$   & $\Z/{6\Z} \times \Z/{2\Z}$ & $1$ & $1$ \\ \hline
41 & $\Q(\sqrt{-753}, \sqrt{-1})$   & $\Z/{6\Z} \times \Z/{2\Z}$ & $1$ & $1$ \\ \hline
42 & $\Q(\sqrt{-813}, \sqrt{-1})$   & $\Z/{6\Z} \times \Z/{2\Z}$ & $1$ & $1$ \\ \hline
43 & $\Q(\sqrt{-913}, \sqrt{-1})$   & $\Z/{6\Z} \times \Z/{2\Z}$ & $1$ & $1$ \\
\hline

\end{longtable*}
\end{center}

\end{minipage}

\begin{minipage}{.46\linewidth}
\begin{center}
\begin{longtable*}[h!]{|c|c|c|c|c|}
\hline
71 & $\Q(\sqrt{-235}, \sqrt{-2515})$ & $\Z/{2\Z}$ & $\Z/{6\Z}$ & $1$ \\ \hline
72 & $\Q(\sqrt{-235}, \sqrt{-3235})$ & $\Z/{2\Z}$ & $\Z/{6\Z}$ & $1$ \\ \hline
73 & $\Q(\sqrt{-403}, \sqrt{-247})$ & $\Z/{2\Z}$ & $\Z/{6\Z}$ & $1$ \\ \hline
74 & $\Q(\sqrt{-403}, \sqrt{-1963})$ & $\Z/{2\Z}$ & $\Z/{6\Z}$ & $1$ \\ \hline
75 & $\Q(\sqrt{-403}, \sqrt{-3523})$ & $\Z/{2\Z}$ & $\Z/{6\Z}$ & $1$ \\ \hline
76 & $\Q(\sqrt{-38}, \sqrt{-6})$ & $\Z/{6\Z}$ & $\Z/{2\Z}$ & $1$ \\ \hline
77 & $\Q(\sqrt{-38}, \sqrt{-22})$ & $\Z/{6\Z}$ & $\Z/{2\Z}$ & $1$ \\ \hline
78 & $\Q(\sqrt{-118}, \sqrt{-6})$ & $\Z/{6\Z}$ & $\Z/{2\Z}$ & $1$ \\
\hline

\end{longtable*}
\end{center}
\end{minipage}
\hfill
\begin{minipage}{0.46\linewidth}
\begin{center}
\begin{longtable*}[h!]{|c|c|c|c|c|}

\hline

109 & $\Q(\sqrt{-29}, \sqrt{-11})$ & $\Z/{6\Z}$ & $1$ & $\Z/{2\Z}$ \\ \hline
110 & $\Q(\sqrt{-29}, \sqrt{-19})$ & $\Z/{6\Z}$ & $1$ & $\Z/{2\Z}$ \\ \hline
111 & $\Q(\sqrt{-29}, \sqrt{-43})$ & $\Z/{6\Z}$ & $1$ & $\Z/{2\Z}$ \\ \hline
112 & $\Q(\sqrt{-29}, \sqrt{-163})$ & $\Z/{6\Z}$ & $1$ & $\Z/{2\Z}$ \\ \hline
113 & $\Q(\sqrt{-53}, \sqrt{-3})$ & $\Z/{6\Z}$ & $1$ & $\Z/{2\Z}$ \\ \hline
114 & $\Q(\sqrt{-53}, \sqrt{-19})$ & $\Z/{6\Z}$ & $1$ & $\Z/{2\Z}$ \\ \hline
115 & $\Q(\sqrt{-53}, \sqrt{-67})$ & $\Z/{6\Z}$ & $1$ & $\Z/{2\Z}$ \\ \hline
116 & $\Q(\sqrt{-61}, \sqrt{-11})$ & $\Z/{6\Z}$ & $1$ & $\Z/{2\Z}$ \\
\hline

\end{longtable*}
\end{center}
\end{minipage}

\newpage
\begin{center}
\begin{longtable*}[h!]{|c|c|c|c|c|}
\hline
No. & $K  = \Q(\sqrt{-x}, \sqrt{-y})$ & $C\ell(K_1)$ & $C\ell(K_2)$ & $C\ell(K_3)$ \\
\hline
147 & $\Q(\sqrt{-37}, \sqrt{-643})$ & $\Z/{2\Z}$ & $\Z/{3\Z}$ & $\Z/{2\Z}$ \\
\hline
148 & $\Q(\sqrt{-37}, \sqrt{-883})$ & $\Z/{2\Z}$ & $\Z/{3\Z}$ & $\Z/{2\Z}$ \\
\hline
149 & $\Q(\sqrt{-37}, \sqrt{-907})$ & $\Z/{2\Z}$ & $\Z/{3\Z}$ & $\Z/{2\Z}$ \\
\hline
150 & $\Q(\sqrt{-26}, \sqrt{-1})$ & $\Z/{6\Z}$ & $1$ & $\Z/{2\Z}$ \\
\hline
151 & $\Q(\sqrt{-106}, \sqrt{-1})$ & $\Z/{6\Z}$ & $1$ & $\Z/{2\Z}$ \\
\hline
152 & $\Q(\sqrt{-202}, \sqrt{-1})$ & $\Z/{6\Z}$ & $1$ & $\Z/{2\Z}$ \\
\hline
153 & $\Q(\sqrt{-298}, \sqrt{-1})$ & $\Z/{6\Z}$ & $1$ & $\Z/{2\Z}$ \\
\hline
\end{longtable*}
\end{center}
\newpage
\noindent
The fields $K$ derived from Table~\ref{tab6} using \texttt{SageMath} are listed in the table below. 
\begin{table}[H]
\centering
\caption{\textbf{ $d_{xy}$ is the product of exactly three distinct primes and $h_K = 6$}}\label{tab13}
\end{table}

\begin{minipage}{.44\linewidth}
\begin{center}
\begin{longtable*}[h!]{|c|c|c|c|c|}
\hline
No. & $\Q(\sqrt{-x}, \sqrt{-y})$ & $C\ell(K_1)$ & $C\ell(K_2)$ & $C\ell(K_3)$ \\
\hline
1 & $ \Q(\sqrt{-15}, \sqrt{-23})$ & $\Z/{2\Z}$ & $\Z/{3\Z}$ & $\Z/{2\Z}$ \\
\hline
2 & $ \Q(\sqrt{-15}, \sqrt{-83})$ & $\Z/{2\Z}$ & $\Z/{3\Z}$ & $\Z/{2\Z}$ \\
\hline
3 & $\Q(\sqrt{-15}, \sqrt{-107})$ & $\Z/{2\Z}$ & $\Z/{3\Z}$ & $\Z/{2\Z}$ \\
\hline
4 & $ \Q(\sqrt{-15}, \sqrt{-283})$ & $\Z/{2\Z}$ & $\Z/{3\Z}$ & $\Z/{2\Z}$ \\
\hline
5 & $ \Q(\sqrt{-15}, \sqrt{-307})$ & $\Z/{2\Z}$ & $\Z/{3\Z}$ & $\Z/{2\Z}$ \\
\hline
6 & $ \Q(\sqrt{-15}, \sqrt{-547})$ & $\Z/{2\Z}$ & $\Z/{3\Z}$ & $\Z/{2\Z}$ \\
\hline
7 & $ \Q(\sqrt{-15}, \sqrt{-643})$ & $\Z/{2\Z}$ & $\Z/{3\Z}$ & $\Z/{2\Z}$ \\
\hline

\end{longtable*}
\end{center}
\end{minipage}
\hfill
\begin{minipage}{0.44\linewidth}
\begin{center}
\begin{longtable*}[h!]{|c|c|c|c|c|}

\hline
36 & $\Q(\sqrt{-115}, \sqrt{-643})$  & $\Z/{2\Z}$ & $\Z/{3\Z}$ & $\Z/{2\Z}$ \\
\hline
37 & $\Q(\sqrt{-115}, \sqrt{-883})$  & $\Z/{2\Z}$ & $\Z/{3\Z}$ & $\Z/{2\Z}$ \\
\hline
38 & $\Q(\sqrt{-115}, \sqrt{-907})$ & $\Z/{2\Z}$ & $\Z/{3\Z}$ & $\Z/{2\Z}$ \\
\hline
39 & $\Q(\sqrt{-123}, \sqrt{-211})$ & $\Z/{2\Z}$ & $\Z/{3\Z}$ & $\Z/{2\Z}$ \\
\hline
40 & $\Q(\sqrt{-123}, \sqrt{-331})$ & $\Z/{2\Z}$ & $\Z/{3\Z}$ & $\Z/{2\Z}$ \\
\hline
41 & $\Q(\sqrt{-123}, \sqrt{-499})$ & $\Z/{2\Z}$ & $\Z/{3\Z}$ & $\Z/{2\Z}$ \\
\hline
42 & $\Q(\sqrt{-123}, \sqrt{-547})$ & $\Z/{2\Z}$ & $\Z/{3\Z}$ & $\Z/{2\Z}$ \\
\hline
43 & $\Q(\sqrt{-123}, \sqrt{-643})$ & $\Z/{2\Z}$ & $\Z/{3\Z}$ & $\Z/{2\Z}$ \\
\hline

\end{longtable*}
\end{center}

\end{minipage}

\begin{minipage}{.46\linewidth}
\begin{center}
\begin{longtable*}[h!]{|c|c|c|c|c|}
\hline
72 & $\Q(\sqrt{-427}, \sqrt{-31})$ & $\Z/{2\Z}$ & $\Z/{3\Z}$ & $\Z/{2\Z}$ \\
\hline
73 & $\Q(\sqrt{-427}, \sqrt{-59})$ & $\Z/{2\Z}$ & $\Z/{3\Z}$ & $\Z/{2\Z}$ \\
\hline
74 & $\Q(\sqrt{-427}, \sqrt{-211})$ & $\Z/{2\Z}$ & $\Z/{3\Z}$ & $\Z/{2\Z}$ \\
\hline
75 & $\Q(\sqrt{-427}, \sqrt{-307})$ & $\Z/{2\Z}$ & $\Z/{3\Z}$ & $\Z/{2\Z}$ \\
\hline
76 & $\Q(\sqrt{-427}, \sqrt{-499})$ & $\Z/{2\Z}$ & $\Z/{3\Z}$ & $\Z/{2\Z}$ \\
\hline
77 & $\Q(\sqrt{-427}, \sqrt{-547})$ & $\Z/{2\Z}$ & $\Z/{3\Z}$ & $\Z/{2\Z}$ \\
\hline
78 & $\Q(\sqrt{-427}, \sqrt{-643})$ & $\Z/{2\Z}$ & $\Z/{3\Z}$ & $\Z/{2\Z}$ \\
\hline
79 & $\Q(\sqrt{-427}, \sqrt{-883})$ & $\Z/{2\Z}$ & $\Z/{3\Z}$ & $\Z/{2\Z}$ \\
\hline
\end{longtable*}
\end{center}
\end{minipage}
\hfill
\begin{minipage}{0.46\linewidth}
\begin{center}
\begin{longtable*}[h!]{|c|c|c|c|c|}

\hline

113 & $\Q(\sqrt{-843}, \sqrt{-19})$ & $\Z/{6\Z}$ & $1$ & $\Z/{2\Z}$ \\
\hline
114 & $\Q(\sqrt{-843}, \sqrt{-67})$ & $\Z/{6\Z}$ & $1$ & $\Z/{2\Z}$ \\
\hline
115 & $\Q(\sqrt{-1059}, \sqrt{-7})$ & $\Z/{6\Z}$ & 1 & $\Z/{2\Z}$ \\
\hline
116 & $\Q(\sqrt{-1059}, \sqrt{-67})$ & $\Z/{6\Z}$ & 1 & $\Z/{2\Z}$ \\
\hline
117 & $\Q(\sqrt{-1059}, \sqrt{-163})$ & $\Z/{6\Z}$ & 1 & $\Z/{2\Z}$ \\
\hline
118 & $\Q(\sqrt{-1099}, \sqrt{-163})$ & $\Z/{6\Z}$ & 1 & $\Z/{2\Z}$ \\
\hline
119 & $\Q(\sqrt{-1147}, \sqrt{-19})$ & $\Z/{6\Z}$ & 1 & $\Z/{2\Z}$ \\
\hline
120 & $\Q(\sqrt{-1147},\sqrt{-163})$ & $\Z/{6\Z}$ & $1$ & $\Z/{2\Z}$ \\
\hline

\end{longtable*}
\end{center}

\end{minipage}

\newpage
\begin{minipage}{.46\linewidth}
\begin{center}
\begin{longtable*}[h!]{|c|c|c|c|c|}
\hline
154 & $\Q(\sqrt{-2283},\sqrt{-7})$ & $\Z/{6\Z}$ & $1$ & $\Z/{2\Z}$ \\
\hline
155 & $\Q(\sqrt{-2283},\sqrt{-11})$ & $\Z/{6\Z}$ & $1$ & $\Z/{2\Z}$ \\
\hline
156 & $\Q(\sqrt{-2283},\sqrt{-43})$ & $\Z/{6\Z}$ & $1$ & $\Z/{2\Z}$ \\
\hline
157 & $\Q(\sqrt{-2283},\sqrt{-163})$ & $\Z/{6\Z}$ & $1$ & $\Z/{2\Z}$ \\
\hline
158 & $\Q(\sqrt{-2443},\sqrt{-43})$ & $\Z/{6\Z}$ & $1$ & $\Z/{2\Z}$ \\
\hline
159 & $\Q(\sqrt{-2443},\sqrt{-163})$ & $\Z/{6\Z}$ & $1$ & $\Z/{2\Z}$ \\
\hline
160 & $\Q(\sqrt{-2515},\sqrt{-3})$ & $\Z/{6\Z}$ & $1$ & $\Z/{2\Z}$ \\
\hline
161 & $\Q(\sqrt{-2515},\sqrt{-7})$ & $\Z/{6\Z}$ & $1$ & $\Z/{2\Z}$ \\
\hline

\end{longtable*}
\end{center}
\end{minipage}
\hfill
\begin{minipage}{0.46\linewidth}
\begin{center}
\begin{longtable*}[h!]{|c|c|c|c|c|}

\hline
195 & $\Q(\sqrt{-1219}, \sqrt{-2})$    & $\Z/{6\Z}$ & $1$ & $\Z/{2\Z}$ \\
\hline
196 & $\Q(\sqrt{-1267}, \sqrt{-2})$    & $\Z/{6\Z}$ & $1$ & $\Z/{2\Z}$ \\
\hline
197 & $\Q(\sqrt{-1315}, \sqrt{-2})$    & $\Z/{6\Z}$ & $1$ & $\Z/{2\Z}$ \\
\hline
198 & $\Q(\sqrt{-1603}, \sqrt{-2})$    & $\Z/6\Z$   & $1$ & $\Z/2\Z$ \\
\hline
199 & $\Q(\sqrt{-1915}, \sqrt{-2})$    & $\Z/6\Z$   & $1$ & $\Z/2\Z$ \\
\hline
200 & $\Q(\sqrt{-1963}, \sqrt{-2})$    & $\Z/6\Z$   & $1$ & $\Z/2\Z$ \\
\hline
201 & $\Q(\sqrt{-2515}, \sqrt{-2})$    & $\Z/6\Z$   & $1$ & $\Z/2\Z$ \\
\hline
202 & $\Q(\sqrt{-2923}, \sqrt{-2})$    & $\Z/6\Z$   & $1$ & $\Z/2\Z$ \\
\hline

\end{longtable*}
\end{center}

\end{minipage}

\begin{minipage}{.46\linewidth}
\begin{center}
\begin{longtable*}[h!]{|c|c|c|c|c|}
\hline

236 & $\Q(\sqrt{-6}, \sqrt{-139})$ & $\Z/{2\Z}$ & $\Z/{3\Z}$ & $\Z/{2\Z}$ \\
\hline
237 & $\Q(\sqrt{-6}, \sqrt{-211})$ & $\Z/{2\Z}$ & $\Z/{3\Z}$ & $\Z/{2\Z}$ \\
\hline
238 & $\Q(\sqrt{-6}, \sqrt{-283})$ & $\Z/{2\Z}$ & $\Z/{3\Z}$ & $\Z/{2\Z}$ \\
\hline
239 & $\Q(\sqrt{-6}, \sqrt{-307})$ & $\Z/{2\Z}$ & $\Z/{3\Z}$ & $\Z/{2\Z}$ \\
\hline
240 & $\Q(\sqrt{-6}, \sqrt{-331})$ & $\Z/{2\Z}$ & $\Z/{3\Z}$ & $\Z/{2\Z}$ \\
\hline
241 & $\Q(\sqrt{-6}, \sqrt{-379})$ & $\Z/{2\Z}$ & $\Z/{3\Z}$ & $\Z/{2\Z}$ \\
\hline
242 & $\Q(\sqrt{-6}, \sqrt{-499})$ & $\Z/{2\Z}$ & $\Z/{3\Z}$ & $\Z/{2\Z}$ \\
\hline
243 & $\Q(\sqrt{-6}, \sqrt{-547})$ & $\Z/{2\Z}$ & $\Z/{3\Z}$ & $\Z/{2\Z}$ \\
\hline
\end{longtable*}
\end{center}
\end{minipage}
\hfill
\begin{minipage}{0.46\linewidth}
\begin{center}
\begin{longtable*}[h!]{|c|c|c|c|c|}

\hline
261 & $\Q(\sqrt{-38}, \sqrt{-11})$ & $\Z/{6\Z}$ & $1$ & $\Z/{2\Z}$ \\
\hline
262 & $\Q(\sqrt{-38}, \sqrt{-43})$ & $\Z/{6\Z}$ & $1$ & $\Z/{2\Z}$ \\
\hline
263 & $\Q(\sqrt{-38}, \sqrt{-67})$ & $\Z/{6\Z}$ & $1$ & $\Z/{2\Z}$ \\
\hline
264 & $\Q(\sqrt{-38}, \sqrt{-163})$ & $\Z/{6\Z}$ & $1$ & $\Z/{2\Z}$ \\
\hline
265 & $\Q(\sqrt{-118}, \sqrt{-3})$ & $\Z/{6\Z}$ & $1$ & $\Z/{2\Z}$ \\
\hline
266 & $\Q(\sqrt{-118}, \sqrt{-11})$ & $\Z/{6\Z}$ & $1$ & $\Z/{2\Z}$ \\
\hline
267 & $\Q(\sqrt{-118}, \sqrt{-19})$ & $\Z/{6\Z}$ & $1$ & $\Z/{2\Z}$ \\
\hline
268 & $\Q(\sqrt{-118}, \sqrt{-43})$ & $\Z/{6\Z}$ & $1$ & $\Z/{2\Z}$ \\
\hline
\end{longtable*}
\end{center}

\end{minipage}

\vspace{2cm}


\noindent
\section{Fields $K$ with $h_K = 7$}
\vspace{1cm}
Note that $h_K = 7$ implies that  $C\ell(K)$ is isomorphic to $\Z/{7\Z}$. Using \texttt{SageMath} for Table \ref{tab7}, we obtain that $h_{K_1} = 1$ and $h_{K_3} = 1$ whenever $h_K = 7$.  Therefore, we omit the columns for $C\ell(K_1)$, $C\ell(K_3)$ and $C\ell(K)$ from the tables presented in this section. All the fields $K$ with $h_K = 7$ are displayed in the table below.
\newpage
\begin{table}[H]
\centering
\caption{\textbf{  $h_K = 7$}}\label{tab14}
\end{table}
\begin{minipage}{0.45\linewidth}
\begin{center}
\begin{longtable}[h!]{|c|c|c|}
\hline
\textbf{No.} & $\Q(\sqrt{-x}, \sqrt{-y})$ & $C\ell(K_2)$ \\
\hline
1 & $\Q(\sqrt{-3}, \sqrt{-447})$   & $\Z/{14\Z}$ \\
\hline
2 & $\Q(\sqrt{-3}, \sqrt{-807})$   & $\Z/{14\Z}$ \\
\hline
3 & $\Q(\sqrt{-3}, \sqrt{-1527})$  & $\Z/{14\Z}$ \\
\hline
4 & $\Q(\sqrt{-3}, \sqrt{-1851})$  & $\Z/{14\Z}$ \\
\hline
5 & $\Q(\sqrt{-3}, \sqrt{-2427})$  & $\Z/{14\Z}$ \\
\hline
6 & $\Q(\sqrt{-3}, \sqrt{-2571})$  & $\Z/{14\Z}$ \\
\hline
7 & $\Q(\sqrt{-3}, \sqrt{-2931})$  & $\Z/{14\Z}$ \\
\hline
\end{longtable}
\end{center}
\end{minipage}
\hfill
\begin{minipage}{0.45\linewidth}
\begin{center}
\begin{longtable}[h!]{|c|c|c|}
\hline

38 & $\Q(\sqrt{-1}, \sqrt{-149})$ & $\Z/{14\Z}$ \\
\hline
39 & $\Q(\sqrt{-1}, \sqrt{-173})$ & $\Z/{14\Z}$ \\
\hline
40 & $\Q(\sqrt{-1}, \sqrt{-349})$ & $\Z/{14\Z}$ \\
\hline
41 & $\Q(\sqrt{-1}, \sqrt{-653})$ & $\Z/{14\Z}$ \\
\hline
42 & $\Q(\sqrt{-1}, \sqrt{-997})$ & $\Z/{14\Z}$ \\
\hline
43 & $\Q(\sqrt{-1}, \sqrt{-1117})$ & $\Z/{14\Z}$ \\
\hline
44 & $\Q(\sqrt{-1}, \sqrt{-1237})$ & $\Z/{14\Z}$ \\
\hline
45 & $\Q(\sqrt{-1}, \sqrt{-1453})$ & $\Z/{14\Z}$ \\
\hline

\end{longtable}
\end{center}
\end{minipage}

\begin{minipage}{0.45\linewidth}
\begin{center}
\begin{longtable}[h!]{|c|c|c|}
\hline
76 & $\Q(\sqrt{-2}, \sqrt{-223})$ & $\Z/{7\Z}$ \\
\hline
77 & $\Q(\sqrt{-2}, \sqrt{-251})$ & $\Z/{7\Z}$ \\
\hline
78 & $\Q(\sqrt{-2}, \sqrt{-463})$ & $\Z/{7\Z}$ \\
\hline
79 & $\Q(\sqrt{-2}, \sqrt{-487})$ & $\Z/{7\Z}$ \\
\hline
80 & $\Q(\sqrt{-2}, \sqrt{-587})$ & $\Z/{7\Z}$ \\
\hline
81 & $\Q(\sqrt{-2}, \sqrt{-811})$ & $\Z/{7\Z}$ \\
\hline
82 & $\Q(\sqrt{-2}, \sqrt{-859})$ & $\Z/{7\Z}$ \\
\hline
83 & $\Q(\sqrt{-2}, \sqrt{-1163})$ & $\Z/{7\Z}$ \\
\hline
\end{longtable}
\end{center}
\end{minipage}
\hfill
\begin{minipage}{0.45\linewidth}
\begin{center}
\begin{longtable}[h!]{|c|c|c|}
\hline
117 & $\Q(\sqrt{-3}, \sqrt{-2011})$ & $\Z/7\Z$ \\
\hline
118 & $\Q(\sqrt{-3}, \sqrt{-2083})$ & $\Z/7\Z$ \\
\hline
119 & $\Q(\sqrt{-3}, \sqrt{-2179})$ & $\Z/7\Z$ \\
\hline
120 & $\Q(\sqrt{-3}, \sqrt{-2251})$ & $\Z/7\Z$ \\
\hline
121 & $\Q(\sqrt{-3}, \sqrt{-2467})$ & $\Z/7\Z$ \\
\hline
122 & $\Q(\sqrt{-3}, \sqrt{-3019})$ & $\Z/7\Z$ \\
\hline
123 & $\Q(\sqrt{-3}, \sqrt{-3067})$ & $\Z/7\Z$ \\
\hline
124 & $\Q(\sqrt{-3}, \sqrt{-3187})$ & $\Z/7\Z$ \\
\hline
\end{longtable}
\end{center}
\end{minipage}

\begin{minipage}{0.45\linewidth}
\begin{center}
\begin{longtable}[h!]{|c|c|c|}
\hline
158 & $\Q(\sqrt{-11}, \sqrt{-587})$ & $\Z/7\Z$ \\
\hline
159 & $\Q(\sqrt{-11}, \sqrt{-811})$ & $\Z/7\Z$ \\
\hline
160 & $\Q(\sqrt{-11}, \sqrt{-827})$ & $\Z/7\Z$ \\
\hline
161 & $\Q(\sqrt{-11}, \sqrt{-859})$ & $\Z/7\Z$ \\
\hline
162 & $\Q(\sqrt{-11}, \sqrt{-1163})$ & $\Z/7\Z$ \\
\hline
163 & $\Q(\sqrt{-11}, \sqrt{-1171})$ & $\Z/7\Z$ \\
\hline
164 & $\Q(\sqrt{-11}, \sqrt{-1627})$ & $\Z/7\Z$ \\
\hline
165 & $\Q(\sqrt{-11}, \sqrt{-1787})$ & $\Z/7\Z$ \\
\hline
\end{longtable}
\end{center}
\end{minipage}
\hfill
\begin{minipage}{0.45\linewidth}
\begin{center}
\begin{longtable}[h!]{|c|c|c|}
\hline
199 & $\Q(\sqrt{-19}, \sqrt{-3067})$ & $\Z/7\Z$ \\
\hline
200 & $\Q(\sqrt{-19}, \sqrt{-3187})$ & $\Z/7\Z$ \\
\hline
201 & $\Q(\sqrt{-19}, \sqrt{-3907})$ & $\Z/7\Z$ \\
\hline
202 & $\Q(\sqrt{-19}, \sqrt{-4603})$ & $\Z/7\Z$ \\
\hline
203 & $\Q(\sqrt{-19}, \sqrt{-5923})$ & $\Z/7\Z$ \\
\hline
204 & $\Q(\sqrt{-43}, \sqrt{-71})$ & $\Z/7\Z$ \\
\hline
205 & $\Q(\sqrt{-43}, \sqrt{-151})$ & $\Z/7\Z$ \\
\hline
206 & $\Q(\sqrt{-43}, \sqrt{-223})$ & $\Z/7\Z$ \\
\hline

\end{longtable}
\end{center}
\end{minipage}

\begin{minipage}{0.45\linewidth}
\begin{center}
\begin{longtable}[h!]{|c|c|c|}
\hline
240 & $\Q(\sqrt{-67}, \sqrt{-1523})$ & $\Z/7\Z$ \\
\hline
241 & $\Q(\sqrt{-67}, \sqrt{-1627})$ & $\Z/7\Z$ \\
\hline
242 & $\Q(\sqrt{-67}, \sqrt{-1787})$ & $\Z/7\Z$ \\
\hline
243 & $\Q(\sqrt{-67}, \sqrt{-1987})$ & $\Z/7\Z$ \\
\hline
244 & $\Q(\sqrt{-67}, \sqrt{-2011})$ & $\Z/7\Z$ \\
\hline
245 & $\Q(\sqrt{-67}, \sqrt{-2083})$ & $\Z/7\Z$ \\
\hline
246 & $\Q(\sqrt{-67}, \sqrt{-2179})$ & $\Z/7\Z$ \\
\hline
247 & $\Q(\sqrt{-67}, \sqrt{-2467})$ & $\Z/7\Z$ \\
\hline
248 & $\Q(\sqrt{-67}, \sqrt{-2707})$ & $\Z/7\Z$ \\ \hline
249 & $\Q(\sqrt{-67}, \sqrt{-3187})$ & $\Z/7\Z$ \\
\hline
250 & $\Q(\sqrt{-67}, \sqrt{-3907})$ & $\Z/7\Z$ \\
\hline
251 & $\Q(\sqrt{-67}, \sqrt{-5107})$ & $\Z/7\Z$ \\
\hline
252 & $\Q(\sqrt{-67}, \sqrt{-5923})$ & $\Z/7\Z$ \\
\hline
253 & $\Q(\sqrt{-163}, \sqrt{-71})$ & $\Z/7\Z$ \\
\hline
254 & $\Q(\sqrt{-163}, \sqrt{-151})$ & $\Z/7\Z$ \\
\hline
255 & $\Q(\sqrt{-163}, \sqrt{-223})$ & $\Z/7\Z$ \\
\hline
256 & $\Q(\sqrt{-163}, \sqrt{-251})$ & $\Z/7\Z$ \\
\hline
257 & $\Q(\sqrt{-163}, \sqrt{-463})$ & $\Z/7\Z$ \\
\hline
258 & $\Q(\sqrt{-163}, \sqrt{-487})$ & $\Z/7\Z$ \\
\hline

\end{longtable}
\end{center}
\end{minipage}
\hfill
\begin{minipage}{0.45\linewidth}
\begin{center}
\begin{longtable}[h!]{|c|c|c|}
\hline
259 & $\Q(\sqrt{-163}, \sqrt{-587})$ & $\Z/7\Z$ \\
\hline
260 & $\Q(\sqrt{-163}, \sqrt{-811})$ & $\Z/7\Z$ \\
\hline
261 & $\Q(\sqrt{-163}, \sqrt{-827})$ & $\Z/7\Z$ \\
\hline
262 & $\Q(\sqrt{-163}, \sqrt{-859})$ & $\Z/7\Z$ \\
\hline
263 & $\Q(\sqrt{-163}, \sqrt{-1163})$ & $\Z/7\Z$ \\
\hline
264 & $\Q(\sqrt{-163}, \sqrt{-1171})$ & $\Z/7\Z$ \\
\hline
265 & $\Q(\sqrt{-163}, \sqrt{-1523})$ & $\Z/7\Z$ \\
\hline
266 & $\Q(\sqrt{-163}, \sqrt{-1627})$ & $\Z/7\Z$ \\
\hline
267 & $\Q(\sqrt{-163}, \sqrt{-1987})$ & $\Z/7\Z$ \\
\hline
268 & $\Q(\sqrt{-163}, \sqrt{-2083})$ & $\Z/7\Z$ \\
\hline
269 & $\Q(\sqrt{-163}, \sqrt{-2251})$ & $\Z/7\Z$ \\
\hline
270 & $\Q(\sqrt{-163}, \sqrt{-2467})$ & $\Z/7\Z$ \\
\hline
271 & $\Q(\sqrt{-163}, \sqrt{-2707})$ & $\Z/7\Z$ \\
\hline
272 & $\Q(\sqrt{-163}, \sqrt{-3067})$ & $\Z/7\Z$ \\
\hline
273 & $\Q(\sqrt{-163}, \sqrt{-3187})$ & $\Z/7\Z$ \\
\hline
274 & $\Q(\sqrt{-163}, \sqrt{-3907})$ & $\Z/7\Z$ \\
\hline
275 & $\Q(\sqrt{-163}, \sqrt{-4603})$ & $\Z/7\Z$ \\
\hline
276 & $\Q(\sqrt{-163}, \sqrt{-5107})$ & $\Z/7\Z$ \\
\hline
277 & $\Q(\sqrt{-163}, \sqrt{-5923})$ & $\Z/7\Z$ \\
\hline

\end{longtable}
\end{center}
\end{minipage}
\vspace{1cm}
\section{Appendix}
    Let $F =\Q(\sqrt{-d})$, where $d$ is a square-free positive integer, be an imaginary quadratic field. For each positive integer $m \leq 100$, one can list all $d$ for which $h_F =m$ using Watkins' article \cite{MW03} and \texttt{SageMath}. We list below all imaginary quadratic number fields with class numbers $1, 2, 3, 4, 6, 7, 8, 12$ and $ 14$ using \cite{ SA92, SA98, HS67, HS75, MW03}. 

\vspace{1cm}

\begin{center}
\begin{longtable}[h!]{|c|p{3.5cm} p{3.5cm} p{3.5cm} p{2.5cm}|}
\hline
\textbf{$h_F$} &  & \textbf{$d$ values} &  &  \\
\hline
1 & \multicolumn{4}{p{14cm}|}{1, 2, 3, 7, 11, 19, 43, 67, 163 } \\
\hline
2 & \multicolumn{4}{p{14cm}|}{5, 6, 10, 13, 15, 22, 35, 37, 51, 58, 91, 115, 123, 187, 235, 267, 403, 427 } \\
\hline
3 & \multicolumn{4}{p{14cm}|}{23, 31, 59, 83, 107, 139, 211, 283, 307, 331, 379, 499, 547, 643, 883, 907 } \\
\hline
4 & \multicolumn{4}{p{14cm}|}{14, 17, 21, 30, 33, 34, 39, 42, 46, 55, 57, 70, 73, 78, 82, 85, 93, 97, 102, 130, 133, 142, 155, 177, 190, 193, 195, 203, 219, 253, 259, 291, 323, 355, 435, 483, 555, 595, 627, 667, 715, 723, 763, 795, 955, 1003, 1027, 1227, 1243, 1387, 1411, 1435, 1507, 1555 } \\
\hline
6 & \multicolumn{4}{p{14cm}|}{26, 29, 38, 53, 61, 87, 106, 109, 118, 157, 202, 214, 247, 262, 277, 298, 339, 358, 397, 411, 451, 515, 707, 771, 835, 843, 1059, 1099, 1147, 1203, 1219, 1267, 1315, 1347, 1363, 1563, 1603, 1843, 1915, 1963, 2227, 2283, 2443, 2515, 2563, 2787, 2923, 3235, 3427, 3523, 3763} \\
\hline
7 & \multicolumn{4}{p{14cm}|}{71, 151, 223, 251, 463, 467, 487, 587, 811, 827, 859, 1163, 1171, 1483, 1523, 1627, 1787, 1987, 2011, 2083, 2179, 2251, 2467, 2707, 3019, 3067, 3187, 3907, 4603, 5107, 5923 } \\
\hline
8 & \multicolumn{4}{p{14cm}|}{41, 62, 65, 66, 69, 77, 94, 95, 105, 111, 113, 114, 137, 138, 141, 145, 154, 158, 165, 178, 183, 205, 210, 213, 217, 226, 238, 258, 265, 273, 282, 295, 299, 301, 310, 313, 322, 330, 337, 345, 357, 371, 382, 385, 395, 418, 438, 442, 445, 457, 462, 466, 478, 498, 505, 553, 562, 577, 579, 583, 598, 651, 658, 697, 742, 793, 862, 915, 939, 979, 987, 995, 1043, 1131, 1155, 1195, 1299, 1339, 1443, 1635, 1651, 1659, 1731, 1771, 1795, 1803, 1939, 1947, 1995, 2035, 2059, 2067, 2139, 2163, 2307, 2323, 2395, 2419, 2451, 2587, 2611, 2667, 2715, 2755, 2827, 2947, 2995, 3003, 3243, 3315, 3355, 3403, 3507, 3595, 3787, 3883, 3963, 4123, 4195, 4267, 4323, 4387, 4747, 4843, 4867, 5083, 5467, 5587, 5707, 5947, 6307 }\\
\hline
\end{longtable}
\end{center}

\end{document}